\documentclass[a4paper, 12pt]{amsart}
\usepackage{amssymb, amsmath, amsthm, color}
\usepackage[colorlinks=true, breaklinks=true, linkcolor=black, citecolor=blue, urlcolor=red]{hyperref} 
\usepackage[english]{babel}
\usepackage{tikz-cd}

\numberwithin{equation}{section}
\newtheorem{lettertheorem}{Theorem}

\newtheorem{lettercorollary}[lettertheorem]{Corollary}

\newtheorem{letterprob}[lettertheorem]{Problem}

\newtheorem{theorem}{Theorem}[section]

\newtheorem{corollary}[theorem]{Corollary}
\newtheorem{proposition}[theorem]{Proposition}
\newtheorem{observation}[theorem]{Observation}
\newtheorem{definition}[theorem]{Definition}
\newtheorem{notation}[theorem]{Notation}
\newtheorem{remark}[theorem]{Remark}
\newtheorem{example}[theorem]{Example}

\newcommand{\act}{\curvearrowright}
\DeclareMathOperator{\ad}{ad}
\newcommand{\al}{\alpha}
\DeclareMathOperator{\Aut}{Aut}
\newcommand{\cC}{\mathcal C}
\newcommand{\C}{\mathbf C}
\newcommand{\fC}{\mathfrak C}
\newcommand{\cF}{\mathcal F}
\DeclareMathOperator{\Frac}{Frac}
\DeclareMathOperator{\FC}{ForC}
\newcommand{\cG}{\mathcal G}
\DeclareMathOperator{\Gr}{Gr}
\newcommand{\Ga}{\Gamma}
\newcommand{\cH}{\mathcal H}
\DeclareMathOperator{\id}{id}
\newcommand{\into}{\hookrightarrow}
\newcommand{\La}{\Lambda}
\newcommand{\la}{\langle}
\DeclareMathOperator{\Leaf}{Leaf}

\DeclareMathOperator{\mcm}{mcm}
\DeclareMathOperator{\Mon}{Mon}
\newcommand{\N}{\mathbf{N}}
\newcommand{\ot}{\otimes}
\newcommand{\Q}{\mathbf{Q}}
\newcommand{\R}{\mathbf{R}}
\newcommand{\ra}{\rangle}
\DeclareMathOperator{\Root}{Root}
\DeclareMathOperator{\Sp}{Spine}
\DeclareMathOperator{\supp}{supp}
\newcommand{\ti}{\tilde}
\newcommand{\cUF}{\mathcal{UF}}
\DeclareMathOperator{\Ver}{Ver}

\newcommand{\Z}{\mathbf{Z}}

\newcommand{\onto}{\twoheadrightarrow}

\begin{document}

\thanks{
AB is supported by the Australian Research Council Grant DP200100067 and a University of New South Wales Sydney Starting Grant.}
\author{Arnaud Brothier}
\address{Arnaud Brothier\\ School of Mathematics and Statistics, University of New South Wales, Sydney NSW 2052, Australia}
\email{arnaud.brothier@gmail.com\endgraf
\url{https://sites.google.com/site/arnaudbrothier/}}

\title[Forest-skein groups II]{Forest-skein groups II: construction from homogeneously presented monoids}

\begin{abstract}
Inspired by the reconstruction program of conformal field theories of Vaughan Jones we recently introduced a vast class of so called forest-skein groups. 
They are built from a skein presentation: a set of colours and a set of pairs of coloured trees.
Each nice skein presentation produces four groups similar to Richard Thompson's group $F,T,V$ and the braided version $BV$ of Brin and Dehornoy.

In this article, we consider forest-skein groups obtained from one-dimensional skein presentations; the data of a homogeneous monoid presentation.
We decompose these groups as wreath products.
This permits to classify them up to isomorphisms. 
Moreover, we prove that a number of properties of the fraction group of the monoid pass through the forest-skein groups such as the Haagerup property, homological and topological finiteness properties, and orderability.
\end{abstract}

\maketitle

\section*{Introduction}
Vaughan Jones has unexpectedly unravelled new connections between subfactors, Richard Thompson's group, conformal field theory, and knot theory in \cite{Jones17}, see also the two surveys \cite{Jones19-survey, Brothier20-survey}.
Inspired by these beautiful links we have initiated a program in the vein of Jones's vision but where the Thompson group is replaced by a collection of so-called {\it forest-skein groups} (in short FS groups) that are constructed from so-called forest-skein (FS) categories \cite{Brothier22-FS}.
These FS categories are in spirit skeletons or discretised versions of certain tensor categories (more precisely Jones' planar algebras) appearing in subfactor theory \cite{Jones99}. 
Moreover, FS categories provide a graphical calculus using trees that resembles the one introduced by Brown to study the Thompson groups \cite{Brown87}. (Note that unpublished manuscripts of Thompson suggests that he was aware of such a description using trees.)
A FS category is the set of binary rooted forests over a vertex-colouring set $S$ that we mod out by some {\it skein relations}. 
These skein relations are expressed by a set $R$ of pairs of coloured trees $(t,s)$ with same number of leaves.
This provides a {\it skein presentation} like the following with two colours $a$ and $b$ and one skein relation
\[\includegraphics{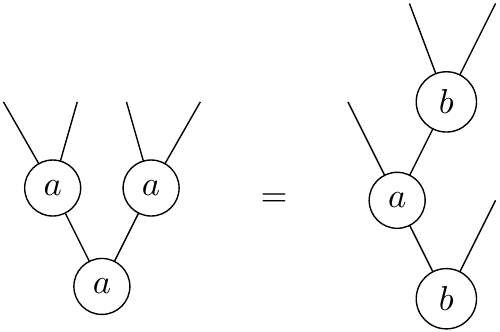}\]
If a FS category $\cF$ admits a {\it calculus of fractions} (is left-cancellative and admits common right-multiples), then we can form a FS group $G_\cF:=\Frac(\cF,1)$ which is the {\it fraction group} of $\cF$. 
Elements of $G_\cF$ are formally of the form $t\circ s^{-1}$ with $t,s$ trees with same number of leaves and where $s^{-1}$ is a formal inverse of $s$.
When such a calculus of fractions exists we say that $\cF$ is a {\it Ore FS category}.
The group $G_\cF$ is similar to the celebrated Richard Thompson group $F$. By adding permutations or braids on top of trees we may form three additional FS groups $G_\cF^T,G_\cF^V,G_\cF^{BV}$ that are analogous to the larger Thompson groups $T,V$ and the braided Thompson group $BV$ (the later was independently discovered by Brin and Dehornoy \cite{Brin-BV1,Dehornoy06}). 
Note that the classical Thompson groups and the braided version are obtained via the skein presentation with one colour and no relations, see \cite{Brown87} or \cite{Cannon-Floyd-Parry96}.

{\bf Properties of FS groups.}
We have proved and observed a number of properties and behaviours of FS groups in \cite{Brothier22-FS}.
For instance, if $X$ is either of the groups $F,T,V,BV$, then $X$ embeds in $G^X_\cF$. In particular, all FS groups are not elementary amenable, have exponential growth, and have infinite geometric dimension (i.e.~all classifying space are infinite dimensional). Moreover, if $X\neq F$, then they all contain a copy of the free group of rank two. 
From a (finite) skein presentation of $\cF$ we can deduce a (finite) group presentation of $G_\cF$.
All FS groups have trivial first $\ell^2$-Betti numbers and an important family of them satisfies the topological finiteness property $F_\infty$ (i.e.~for every $n\geq 1$ there exists a classifying space with finite $n$-skeleton).
Although, we are still far from understanding which kind of groups we may obtain, how much they differ from the Thompson groups, and how much they remember their skein presentation. 
For instance, are they simple groups (or contain {\it obvious} subgroups that are simple), what kind of analytic and geometric properties do they satisfy (do they all have the Haagerup property), what kind of topological properties do they have (are they all either not finitely presented or of type $F_\infty$), do they all have nice dynamical properties like having faithful actions on totally ordered sets?

{\bf Purpose and main result.}
This article has for purpose to construct explicit examples from which we may witness easily which properties they satisfy and have a classification; thus answering some of the question of above. 
We consider skein presentations that are one-dimensional, i.e.~presentations $(S,R)$ where $R$ is a set of pairs of trees $(t,s)$ where $t,s$ are left-vines like the following: $S=\{a,b\}$ and a single skein relation given by the two trees:
\[\includegraphics{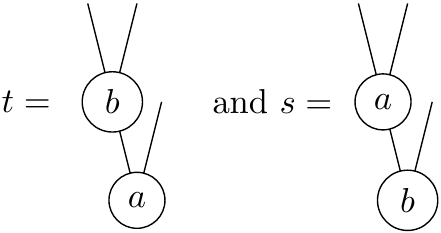}\]
It is not hard to see that these skein presentations are in bijection with {\it homogeneously presented (HP) monoids} (the monoid relations are pairs of words of same length).
The example of above corresponds to the monoid presentation $\Mon\la a,b| ab=ba\ra$: the free abelian monoid in two generators.
We obtain a functor $M\to \cF_M$ from HP monoids to presented FS categories.
Moreover, we observe that $M$ admits a calculus of fractions (we say that $M$ is a {\it Ore monoid}) if and only if $\cF_M$ does. 
This allows us to easily construct a huge class of FS groups using classical monoids.
The main theorem of the paper is an explicit decomposition of the FS group of $\cF_M$ in terms of the fraction group of $M$.
Before stating it we introduce some notations.

Let $V\act \Q_2$ be the usual action of Thompson's group $V$ on the set of dyadic rational in the unit torus $\Q_2:=\Z[1/2]/\Z.$
Extend this action to $BV$ via the canonical map $Q:BV\onto V$.
Elements of $V$ admits right-derivatives that we denote by $v'(q)$ for $v\in V, q\in \Q_2.$
If $w\in BV$, then we write $w'(q)$ for $Q(w)'(q)$.
Finally, $\log_2$  denotes the base-$2$ logarithm, i.e.~$\log_2(2^n)=n$.

\begin{lettertheorem}\label{lettertheo:decomposition}(Theorem \ref{theo:FCWP})
Consider a Ore HP monoid $M=\Mon\la \sigma|\rho\ra$ with fraction group $\Frac(M)$, a generator $a\in \sigma$, the associated FS category $\cF$ and FS groups $G^X$ with $X=F,T,V,BV$.
Write $K_M\lhd \Frac(M)$ for the kernel of the word-length map $\Frac(M)\to\Z.$

There is a group isomorphism
$$\Frac(\cF,1)\simeq K_M\wr_{\Q_2,\tau}X$$
where $K_M\wr_{\Q_2,\tau}X:=\oplus_{\Q_2}K_M\rtimes X$ is the twisted permutational restricted wreath product described by the following formula:
$$(v\cdot k)(q):= a^{-\log_2(v'(v^{-1}q))}\cdot k(v^{-1} q) \cdot a^{\log_2(v'(v^{-1}q))}$$
for $v\in X, k\in \oplus_{\Q_2}K_M, q\in \Q_2.$

In particular, if the generator $a\in\sigma$ is central in $\Frac(M)$, then the twist is trivial.
\end{lettertheorem}

The twist of the wreath product is the cocycle map
$$\tau:X\times \Q_2\to \Aut(K_M),\ (v,q)\mapsto \ad(a^{-1})^{\log_2(v'(v^{-1} q))}.$$
We will say that we are in the untwisted case when $\tau(v,q)=\id$ for all $(v,q)\in X\times \Q_2$, i.e.~the generator $a$ is central in $\Frac(M)$.
Note that in general $G^X$ is not isomorphic to the obvious candidate $\Frac(M)\wr_{\Q_2}X$. Indeed, we will prove that the wreath product $K_M\wr_{\Q_2,\tau}X$ completely remembers $K_M$.
Therefore, if $K_M\not\simeq \Frac(M)$, then $G^X\not\simeq \Frac(M)\wr_{\Q_2}X.$ This becomes clear in the proof of the theorem where we separate the copy of the Thompson groups as pairs of monochromatic trees and the copy of $\oplus_{\Q_2}K_M$ which corresponds to a single tree but with two different vertex-colouring. 
A convincing and elementary example is the Thompson group $F$ which is isomorphic to the fraction group of the FS category with one colour and no skein relations. The underlying monoid is $M=\N$ with fraction group $\Frac(M)=\Z$ and kernel $K_M=\{e\}$: the trivial group. The group $F$ is indeed isomorphic to $\{e\}\wr_{\Q_2}F$ but not isomorphic to $\Z\wr_{\Q_2}F.$

This specific group structure has been previously witnessed in our work on Jones' actions \cite{Brothier22}. 
Indeed, these groups appear when one considers covariant monoidal functors from the category of monochromatic forests into the category of groups and when a caret is the map $g\mapsto (\alpha(g),e)$ from a group to its square and where $\alpha$ is a group automorphism.
We previously observed that these groups are isomorphic to some of the groups constructed by Tanushevski and by Witzel-Zaremsky using pure cloning systems \cite{Tanushevski16,Tanushevski17,Witzel-Zaremsky18,Zaremsky18-clone}.
We refer the reader to the appendix of \cite{Brothier21} for a detailed discussion and comparison of these three constructions.

From this structural theorem we deduce a number of corollaries. 
\begin{lettercorollary}\label{cor:B}
Consider two Ore HP monoids $M$ and $\ti M$ with associated FS groups $G^X$ and $\ti G^X$ with $X=F,T,V$.
In the general twisted case we have that
$$\text{ if } G^X\simeq \ti G^X \text{ for one } X\in\{F,T,V\} \text{, then } K_M\simeq K_{\ti M}.$$
The converse holds in the untwisted case.
\end{lettercorollary}
The case $X=V$ is easily deduced from \cite[Theorem 4.12]{Brothier22} (in section 5 of this later article we even described precisely all isomorphisms between two such groups in the untwisted case). 
We provide two independent short proofs for the $F$ and $T$-cases that are based on two different tricks, see Theorem \ref{theo:classifWP}.
Our proofs are heavily using the simplicity of $V,T$ and the derived subgroups of $F$. It is not clear how to adapt our approach to the group $BV$ which has a much richer normal subgroup structure, see for instance \cite{Zaremsky18-BV}.

Note that each FS group $G$ constructed from a Ore HP monoid is again the fraction group of a larger HP monoid $\cF_\infty$ obtained by considering infinite but finitely supported lists of trees, see Section \ref{sec:FS-categories}. We may iterate this construction and take the FS group obtained from the monoid $\cF_\infty$.
Hence, one single Ore HP monoids permits to construct infinitely many new ones with associated fraction groups that we may show (in certain cases) to be pairwsie non-isomorphic using the theorem of above.

We then investigate properties of different flavour: the analytic or geometric property of Haagerup (equivalent to Gromov's a-T-menability), homological and topological finitness properties, and orderability.
We refer to Sections \ref{sec:Haagerup}, \ref{sec:finiteness}, and \ref{sec:orderability} for definitions, references, and details on these properties. 

\begin{lettercorollary}\label{lettercor}
Consider a Ore HP monoid $M$ with fraction group $\Frac(M)$ and associated FS groups $G^X$ with $X=F,T,V,BV.$
The following assertions are true.
\begin{enumerate}
\item If $X\neq BV$, then $\Frac(M)$ has the Haagerup property if and only if $G^X$ does;
\item For all $X$ and $m\in\N^*\cup\{\infty\}$, if we are in the untwisted case and $\Frac(M)$ satisfies the homological finiteness property $FP_m$ (resp.~the topological finiteness property $F_m$), then so does $G^X$;
\item For $X=F,T,V$ and $m\in\N^*\cup\{\infty\}$, if $\Frac(M)$ satisfies the topological finiteness property $F_m$, then so does $G^X$;
\item We have that $\Frac(M)$ is left-orderable if and only if $G$ or $G^{BV}$ is.
\item We have that $\Frac(M)$ is bi-orderable if and only $G$ is.
\end{enumerate}
\end{lettercorollary}

Note that the choice of investigating these specific properties is not at all random. These are properties that one expects for Thompson-like groups. Moreover, the extra structure of the underlying FS category is well-adapted for investigating each of these specific properties as we are about to explain.

The first statement was previously proven in greater generality for the $V$-case \cite{Brothier19WP}. 
Since the Haagerup property is closed under taking subgroup we immediately deduce the $F$ and $T$-cases.
It is still unknown today if all braid groups have the Haagerup property and thus unknown for the braided Thompson group $BV$ and even less for the forest-skein generalisations.

In \cite{Brothier19WP} we gave the first analytic but not geometric proof showing that certain wreath products had the Haagerup property. It permitted to include examples that were previously out of reach from other known geometrical approaches.
The proof extended non trivially some novel techniques for constructing positive definite maps with good asymptotic properties. Techniques which were introduced by Jones and the author in \cite{Brothier-Jones19}, see also \cite[Section 2.2.2]{Brothier20-survey} for the story.
We strongly believe that a similar technology can be applied to prove that many more FS groups have the Haagerup property.

The homological and topological statements in the untwisted case are direct corollaries of a theorem of Bartholdi, Cornulier, and Kochloukova on untwisted wreath products \cite{Bartholdi-Cornulier-Kochloukova15}. Their proof is purely homological and quasi-optimal: there is a converse for groups with infinite abelianisation, see Section \ref{sec:finiteness} for details.
The topological statement in the general twisted case is proved by constructing an explicit classifying space and following a rather standard method first used by Brown \cite{Brown87}.
We have adapted a clever construction of a classifying space due to Tanushevski (that is unpublished) which is sketched by Witzel and Zaremsky in \cite[Section 6]{Witzel-Zaremsky18}. We explain the construction and the main idea of this topological proof in Section \ref{sec:finiteness}.

Note that we have previously proved that a large class of FS groups are of type $F_\infty$ using a technique due to Thumann that is purely topological, see \cite{Brothier22-FS,Thumann17}. This last class is rather different from the one studied in this article. However, the two classes share a common infinite intersection. 

Most proofs on finiteness properties of Thompson-like groups are obtained by constructing a nice classifying space using some extra structure and then to apply the so-called Brown's criterion. It is less common to have a homological proof (thus using resolutions of $\Z[G]$-modules). Although, the wreath product structure allows to do so.
It is very unclear at the moment how to apply such homological techniques to other FS groups or Thompson-like groups. 
Nevertheless, we suggest natural candidates for which such an extension may hold, see the end of Section \ref{sec:finiteness}.

The statement on orderability follows easily from the wreath product structure since $F$ is bi-orderable and $BV$ is left-orderable (but not bi-orderable) by Burillo and Gonzalez-Meneses \cite{Burillo-Gonzalez-Meneses08}, see also Ishida \cite{Ishida18}. 

To any FS group $G$, we have previously introduced a canonical totally ordered set $(Q,\leq)$ on which $G$ acts in an order-preserving way, see Section \ref{sec:canonical} and \cite[Section 6]{Brothier22}.
We have initially thought that the action $G\act Q$ was always faithful which we prove in this article to be false.
The fact that some $G$ are not left-orderable shows in addition that $Q$ cannot be replaced by a more sophisticated totally ordered set on which $G$ acts faithfully (in that case $G$ would be left-orderable by pulling back the order of this space). 
In particular, it shows how different are our group that are constructed {\it diagrammatically} compare to a number of Thompson-like groups that are constructed {\it dynamically} and usually via an order-preserving action (like a piecewise affine one).
This also provides additional evidences that FS groups are rather different than Guba-Sapir's diagram groups. 
Indeed, these latter groups are all bi-orderable \cite{Guba-Sapir06}, see also \cite{Wiest03}.
In a future article of Ryan Seelig and the author we will prove that the assumption of having a faithful action $G\act (Q,\leq)$ is intimately linked with the simplicity of its derived subgroup and on a simplicity statement of its underlying FS category \cite{Brothier-Seelig23}.

Our Corollary \ref{lettercor} provides explicit examples of FS groups that are left-orderable but not bi-orderable and examples that are not even left-orderable. 
Although, the following problems remain unanswered.

\begin{letterprob}
\begin{enumerate}
\item Is there a FS group that does not have the Haagerup property?
\item Is there a finitely presented FS group that is not of type $F_\infty$?
\end{enumerate}
\end{letterprob}

Observe that the first problem is equivalent to ask if all $\Frac(M)$ (with $M$ an Ore HP monoid) has the Haagerup property. Indeed, the class of groups with the Haagerup property is closed under taking subgroups. Moreover, $K_M\subset\Frac(\cF_M,1)\subset \Frac((\cF_M)_\infty)$ where $K_M$ is the kernel of the length map of $\Frac(M)$, $\cF_M$ is the FS category deduced from $M$, and $(\cF_M)_\infty$ is the FS {\it monoid} constructed from the category $\cF_M$ which is itself a Ore HP monoid.
To answer positively the second problem it would be sufficient to produce a fraction group $\Frac(M)$ that is finitely presented but not of type $F_\infty$.

{\bf Plan.} 
Section \ref{sec:framework} introduces notations and the setting: homogeneously presented (HP) monoids, forest-skein (FS) categories and groups. 
In Section \ref{sec:shape} we introduce {\it shape-preserving} skein presentations: all skein relations are pairs of trees $(t,s)$ where $t,s$ only differ by their vertex-colouring (and thus have same {\it shape}).
Note that all FS categories obtained from monoids have a shape-preserving skein presentation. We provide a semi-direct product decomposition for their associated FS groups using certain colouring and decolouring maps.
In Section \ref{sec:decomposition} we decompose fraction groups of HP monoids and their associated FS groups. This gives Theorem \ref{lettertheo:decomposition}.
Section \ref{sec:corollaries} investigates consequences of Theorem \ref{lettertheo:decomposition}. 
We start by proving Corollary \ref{cor:B}.
We then introduce and discuss about the Haagerup property, homological and topological finiteness properties, orderability and prove Corollary \ref{lettercor}.
In Section \ref{sec:canonical} we recall the construction of a canonical action on a totally ordered space for each FS group. We prove that this action is rather trivial when the FS group is obtained from a monoid or more generally from a shape-preserving skein presentation.
Section \ref{sec:examples} presents examples of FS groups and provides a bank of counterexamples for various properties that we may or not expect for FS groups.

{\bf Acknowledgement:}
We warmly thank Matt Zaremsky for stimulating and pertinent exchanges regarding forest-skein groups.
We are also grateful to Mark Kambites and Robert Gray for making helpful comments regarding the first problem and to point out relevant references.
Finally, we are very grateful to Christian de Nicola Larsen and Ryan Seelig for several discussions and their careful readings which revealed imprecisions on a previous version of this article.

\section{General framework}\label{sec:framework}
We introduce notations and terminology for monoids and their presentation.
Then we briefly recall what forest-skein (FS) categories are. 
We explain how a homogeneously presented (HP) monoid permits to define a FS category. 
We then recall what a calculus of fractions is and explain when and how monoids and FS categories produce fraction groups.

In this article we assume that the reader has some basic knowledge on the three Richard Thompson's groups $F\subset T\subset V$ and their Brown's tree-diagrammatic description. For background on those we recommend the seminal paper of Brown and the expository article of Cannon, Floyd, and Parry \cite{Brown87, Cannon-Floyd-Parry96}.
We will also assume the reader knows what the braided Thompson group $BV$ is. We will only use very basic fact about it: its diagrammatic description and the usual map $BV\onto V$.
We recommend the two foundational articles of Brin on that subject \cite{Brin-BV1, Brin-BV2}.
We will recall what FS categories and groups are but for additional information we recommend our previous paper \cite{Brothier22-FS}.

\subsection{Homogeneous relations, presentations, and monoids}

{\bf Free monoid.}
If $\sigma$ is a {\it nonempty} set we write $\Mon\la\sigma\ra$ for the \textit{free monoid} over $\sigma$: elements of $\Mon\la\sigma\ra$ are finite sequences over $\sigma$ that we compose by concatenation from left to right. 
The empty sequence is allowed and denoted $e$. It is the neutral or trivial element of $\Mon\la\sigma\ra.$
If $\sigma=\{a_1,\cdots,a_n\}$, then we write $\Mon\la a_1,\cdots,a_n\ra$ for $\Mon\la\sigma\ra.$

{\bf Letter and word.}
An element of $\sigma$ may be called a \textit{letter, colour,} or a \textit{generator} of $M$.
A finite sequence over $\sigma$ is called a \textit{word} and usually denoted $b_1b_2\cdots b_n$ rather than $(b_1,\cdots,b_n)$ for $b_i\in\sigma.$

{\bf Length.}
Define the length $|w|=\lambda(w)$ of a word $w$ to be the number of letters composing it. Hence, $|e|=0$ and $|b_1\cdots b_n|=n$.
This defines a monoid morphism 
$$\lambda:\Mon\la\sigma\ra\to \N,\ w\mapsto |w|$$ where $\N=\{0,1,2\cdots\}$ is equipped with the usual addition $+$.

{\bf Homogeneous relation and presentation.}
An \textit{homoneous relation} over $\sigma$ is a pair of words $(u,v)$ such that $u$ and $v$ have the same length.
Let $\rho$ be a set of homogeneous relations over $\sigma$.
We say that $\pi:=(\sigma,\rho)$ is a \textit{homogeneous (monoid) presentation} (in short HP) and write 
$$\Mon\la\pi\ra \text{ or } \Mon\la\sigma|\rho\ra$$ for the associated monoid, i.e.~the quotient of $\Mon\la\sigma\ra$ by the equivalence relation generated by $\rho$.
We say that $\Mon\la\pi\ra$ is a \textit{homogeneously presented monoid} or in short a HP monoid.
In practice, we may write $u=v$ to express a relation $(u,v)$ and remove brackets for expressing $\sigma$ in the presentation of a monoid. For instance we may write
$$\Mon\langle a,b| ab=ba\rangle$$
for $\Mon\langle \{a,b\} | (ab,ba)\rangle.$
Most of our explicit monoids will have two generators and one or two relations but nothing prevent to have presentations of any cardinal.
\begin{remark}
Consider a homogeneous presentation $\pi=(\sigma,\rho)$ and its associated monoid $M$.
Since all relations have same length we can factorise the length morphism $\lambda:\Mon\la\sigma\ra\to\N$ into a monoid morphism:
$$\lambda_\pi:M\to \N.$$
Since $\sigma$ is assumed to be nonempty we have that $\lambda_\pi$ is surjective.
\end{remark}

\subsection{General forest-skein categories}\label{sec:FS-categories}

{\bf Tree, forest, free forest-skein category.}
We briefly recall what a forest-skein category is. We refer the reader to \cite{Brothier22-FS} for extensive details.
We call {\it tree} a finite rooted binary tree $t$ so that all vertices have zero or two children. We represent a tree as a graph embedded in the plane with the root on the bottom and leaves on top. Edges either go toward top left or top right and are called left-edges and right-edges.
The vertices without children are called leaves and are ordered from left to right and numbered starting at $1$. A coloured tree is a pair $(t,c)$ where $c:\Ver(t)\to S$ is a map from the interior vertices of $t$ (the vertices that are not leaves). Hence, $(t,c)$ is the data of a tree and a colouring of its (interior) vertices. We often write $t$ rather than $(t,c)$ and drop the word ``coloured".
Here is an example of a tree with three interior vertices coloured by $a$ and $b$ and four leaves:
\[\includegraphics{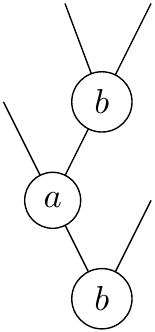}\]
A forest $f$ is a finite list of trees $(f_1,\cdots,f_n)$. We write $\Root(f),\Leaf(f)$ its sets of roots and leaves, respectively.
We represent them in the plane with $f_j$ placed to the left of $f_{j+1}$.
Vertical stacking of forests defines a partially defined associative binary operation. We write $f\circ g$ for the forest obtained by placing $g$ on top of $f$ when $f,g$ are forests satisfying $|\Leaf(f)|=|\Root(g)|.$

\begin{definition}
The collection of all such forests equipped with the composition forms a small category $\cUF\la S\ra$ that we call the \emph{free forest-skein category} over the set $S$.
When $S=\{x\}$ is a singleton, then we write $\cUF\la x\ra$ for $\cUF\la S\ra$ and call it is the \emph{monochromatic} free forest-skein category with colour $x.$
\end{definition}

If $f\ot g$ denotes the {\it horizontal} stacking consisting in placing $f$ to the left of $g$, then we obtain an associative binary operation that is compatible with the composition $\circ$.
Note that if $f=(f_1,\cdots,f_n)$ and $g=(g_1,\cdots,g_m)$ are written as two lists of trees, then 
$$f\ot g=(f_1,\cdots,f_n, g_1,\cdots,g_m)$$ is the concatenation of these lists.
We call the operation $\ot$ the tensor product of $\cUF\la S\ra$. 

{\bf Skein presentation, forest-skein category.}
A skein presentation is a pair $(S,R)$ where $S$ is a nonempty set called the {\it colour} set and $R$ is a set of pairs $(t,s)$ where $t,s$ are coloured trees with same number of leaves.
This defines a {\it forest-skein category} $\cF:=\FC\la S|R\ra$ which is the quotient of $\cUF\la S\ra$ by the equivalence relation $\sim$ generated by the skein relations.
That is $t\sim s$ for all $(t,s)\in R$ and $\sim$ is closed under taking composition and tensor product.
Note that the two operations $\circ,\ot$ both extend to $\cF$. If we add the empty diagram (in order to have a unit for $\ot$) we obtain a small tensor category.
We continue to call {\it forests} the elements of $\cF$.
Here is an example of a skein presentation with two colours $a,b$ and one skein relation
\[\includegraphics{t_as_b.pdf}\]

{\bf Forest-skein monoid.}
Given $(S,R)$ we also form a {\it forest-skein monoid} $\cF_\infty$ which consists in infinite sequences of trees $(f_1,f_2,\cdots)$ of $\cF$.
We require that $f_j=I$ is the trivial tree for all but finitely many $j\geq 1$. 
We extend the composition $\circ$ to $\cF_\infty$ obtaining a monoid.

{\bf Notation.}
We typically use the symbols $a,b$ for colours, $f,g,h$ for forests, and $t,s$ for trees.
A splitting is called a {\it caret} and if this caret is coloured by say $b$ we call it a $b$-caret.
We may write it $Y_b$ and identifying it with a tree with two leaves.
We write $I$ for the trivial tree: the tree with one root and one leaf (this leaf being equal to the root). It is always uncoloured.
We write $b_{j,n}$ for the so-called {\it elementary forest} that has $n$ roots, all its trees are trivial except the $j$-th one which is a single $b$-caret.
Using tensor product notations we obtain
$$b_{j,n}= I^{\ot j-1} \ot Y_b \ot I^{\ot n-j}.$$
Similarly, we write $b_j$ for the infinite forest of $\cF_\infty$ having a single caret at the $j$-th root that is coloured by $b$.

{\bf Terminology.} We will often use the acronym FS for ``forest-skein". Hence, we will write FS category and FS monoid for forest-skein category and forest-skein monoid. Similarly, we will later write FS group for forest-skein group.

\subsection{Forest-skein categories associated to monoids}

{\bf FS category and skein presentation.}
For any homogeneously presented monoid $M=\Mon\langle \sigma|\rho\rangle$ we associate a FS category $\cF_M$ and a skein presentation that we now describe.
Consider the free FS category with set of colours $S=\sigma$.
Given any word $u=a_{1}\cdots a_{n}\in M$ with $a_j\in \sigma$ we consider the tree $\tilde u$ with $n+1$ leaves with all its interior vertices between the root and its first leaf (the left leaf) that are of colour $a_{1},\cdots, a_{n}$  when we read them from the root to the leaf, 
$$\text{i.e.~the tree } \ti u=(a_1)_{1,1}\circ (a_2)_{1,2}\circ\cdots \circ (a_n)_{1,n},$$ 
i.e.~the first leaf of $\tilde u$ has for branch colour the word $u$. 
We may say that $\ti u$ is a {\it left-vine} with $n$ carets and of colour $u$.
For instance, if $\sigma=\{a,b\}$ and we consider the word $u=aab\in M$ we have that $\ti u$ is equal to the following tree:
\[\includegraphics{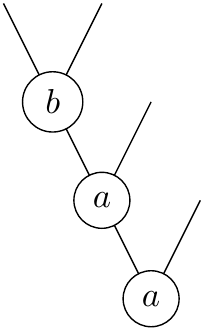}\]
This is a left-vine with three carets and of colour $aab.$
Given a relation $r:=(u,v)\in \rho$ we consider the associated pair of trees $\ti r:=(\ti u,\ti v).$
This defines a skein presentation $P:=(S, R)$ where $S=\sigma$ and $R$ is the family of relations $\ti r$.
We write $\cF_M$ for the associated FS category or simply $\cF$ if the context is clear.

\begin{example}
Consider the HP monoid $M=\Mon\la a,b| aba=bab\ra$.
The associated FS category $\cF_M$ has the following skein presentation with two colours and one skein relation
\[\includegraphics{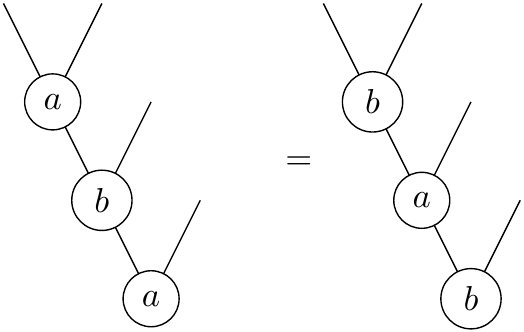}\]
\end{example}

{\bf Forest monoid and presentation.}
We write $\cF_\infty=(\cF_M)_\infty$ for the FS monoid associated to $\cF=\cF_M$.
The monoid $\cF_\infty$ admits a monoid-presentation deduced from the skein presentation of $\cF$ as observed and described in \cite[Section 4]{Brothier22-FS}.
We remind the reader what is this presentation in this specific case.
The monoid $\cF_\infty$ admits the set of generators 
$$S_\infty =( a_j:\ a\in S, j\geq 1)=S\times \N^*$$ 
where $a_j$ corresponds to the elementary forest having only trivial trees except the $j$-th one that has two leaves and one (interior) vertex coloured by $a$.
Consider a relation $r=(u,v)$ of $M$.
We have constructed a pair of trees $\ti r=(\ti u,\ti v)$ from it.
Now, consider the pair of forests $\ti r_j=(\ti u_j,\ti v_j)$ for $j\geq 1$ where $\ti u_j$ is the forest having only trivial trees except the $j$-th one that is equal to $\ti u$.
If $u=a_1\cdots a_\ell$, then $\ti u_j=(a_1)_j (a_2)_j\cdots (a_\ell)_j$ and this spelling is unique in letters of $S_\infty$. 
Hence, $\ti r_j:=(\ti u_j,\ti v_j)$ provides a relation on the free monoid generated by $S_\infty$.
Consider now the Thompson-like relations inherent to all forest-skein monoids:
$$x_q y_j = y_j x_{q+1} ,\ x,y\in S,\ 1\leq j<q.$$
If $R_\infty$ is the union of all these relations we obtain that $P_\infty:=(S_\infty,R_\infty)$ is a monoid-presentation of $\cF_\infty$. 
Note that this presentation is canonically deduced from the presentation $(\sigma,\rho)$ of $M$.

{\bf Notation.}
If $M$ is a HP monoid as above, then $\cF_M$ and $(\cF_M)_\infty$ stand for the associated FS category and FS monoid, respectively.

Here is a rather obvious observation.

\begin{observation}
If $M=\Mon\langle \sigma|\rho\rangle$ is a HP monoid, then so is the FS monoid $(\cF_M)_\infty$ with the associated presentation described above.
\end{observation}

\begin{example}
Consider again the HP monoid
$$M=\Mon\la a,b| aba=bab\ra.$$
Its associated FS monoid admits the following infinite presentation:
\begin{itemize}
\item with generators $a_1,b_1,a_2,b_2,\cdots$;
\item and relations
\begin{enumerate}
\item $a_jb_ja_j=b_ja_jb_j$ for all $j\geq 1$;
\item $x_q y_i = y_i x_{q+1}$ for all $1\leq i<q$ and $x,y\in \{a,b\}.$
\end{enumerate}
\end{itemize}

\end{example}

\subsection{Calculus of fractions}

We now explain how to construct groups and groupoids from certain monoids and categories.

\begin{definition}
A monoid $M$ 
\begin{enumerate}
\item is {\it left-cancellative} if $xy=xy'$ implies $y=y'$ for all $x,y,y'\in M$;\\
\item satisfies {\it Ore's property} if for all $x,y\in M$ there exists $u,v\in M$ so that $xu=yv$, i.e.~any pair of elements of $M$ admits a common right-multiple;\\
\item is a {\it Ore monoid} or admits a {\it (right-)calculus of fractions} if it is left-cancellative and satisfies Ore's property.
\end{enumerate}
\end{definition}

{\bf Convention.} We will exclusively consider {\it right}-calculus of fractions. 

{\bf Fraction group of a monoid.}
When $M$ is a Ore monoid we can form its {\it fraction group} $\Frac(M)$ as follows.
Write $\Frac(M)$ for the set $M\times M$ quotiented by the equivalence relation $\sim$ generated by
$$(x,y)\sim (xz,yz) \text{ for all } (x,y)\in M\times M, z\in M.$$
Write $[x,y]$ for the equivalence class of $(x,y)$ in the quotient $\Frac(M)$.
The set $\Frac(M)$ admits the following group structure with multiplication
$$[x,y]\cdot [x',y'] = [xz,y'z'] \text{ where } yz=x'z'$$
and inverse $[x,y]^{-1}=[y,x]$.

Note that $x\mapsto [x,e]$ is injective only when $M$ is {\it right}-cancellative as well (on top of being {\it left}-cancellative).
Nevertheless, it is very useful and convenient to interpret $[x,y]$ as $x\circ y^{-1}$.
In that description $\Frac(M)$ is generated by elements of $M$ and formal inverses of it.

We recall some basic facts on $\Frac(M)$.
\begin{remark}
Let $M$ be a Ore monoid with fraction group $\Frac(M)$.
\begin{enumerate}
\item If $\pi$ is a a monoid presentation of $M$, then it is a group presentation of $\Frac(M)$ as well. We may write $\Frac(M)=\Gr\la\pi\ra$ to indicate that we consider the {\it group} (and note the monoid) obtained from the presentation $\pi$.
\item If $M=\Mon\la\pi\ra$ is a Ore HP monoid, then the length map factorises into a surjective group morphism
$$\lambda_\pi:\Frac(M)\to\Z.$$
In particular, $\Frac(M)$ does not have Kazhdan property (T) and is never simple.
\item If $M$ is a HP monoid, then for all $x\in M, x\neq e$ and $n\geq 1$ we have $x^n\neq e$. This does not mean that its fraction group (when it exists) has no torsion.
Consider $$M=\Mon\langle a,b| ab=ba, aa=bb\rangle.$$
It satisfies Ore's property since it is abelian.
To show that $M$ is left-cancellative it is sufficient to prove that the two maps 
$$L_a:M\to M, x\mapsto ax \text{ and } L_b:M\to M, x\mapsto bx$$ are injective. 
By symmetry of the presentation it is sufficient to do it for $L_a$.
Any element of $M$ can be written uniquely as: $a^n$ or $a^nb$ with $n\geq 0$.
Hence, there is exactly two elements of $M$ of each length $n\geq 1$.
This implies that $L_a$ is injective since it restricts to a bijection from the set of words of length $n$ to the set of words of length $n+1$.
This proves that $M$ is left-cancellative.
Now, its fraction group $G$ has same presentation $\Gr\langle a,b| ab=ba, aa=bb\rangle.$
Observe that $ab^{-1}$ is an element of $G$ of order two.
\end{enumerate}
\end{remark}

\subsection{Forest-skein groups}
If $\cF$ is FS category (or more generally any small category) we can similarly define left-cancellativity, Ore's property, and Ore FS category. The only difference being that we only consider composable pairs in the definitions. For instance, $\cF$ is left-cancellative if given any triple of forests $(f,g,g')$ satisfying that $|\Leaf(f)|=|\Root(g)|=|\Root(g')|$ and $f\circ g=f\circ g'$ we have $g=g'$.

{\bf Fraction groupoid.}
We denote by $\Frac(\cF)$ the fraction groupoid of a Ore FS category $\cF$. 
Hence, $\Frac(\cF)$ is the set of $[f,g]=f\circ g^{-1}$ with $f,g$ forests with same number of leaves. 
Equivalently, $\Frac(\cF)$ is the set of classes $[f,g]$ of pairs $(f,g)$ under the equivalence relation $\sim$ generated by 
$$(f,g)\sim (f\circ p,g\circ p) \text{ for all composable forest } p.$$
We think of $[f,g]$ as the diagram of $f$ with on top of it the diagram of $g$ but drawn upside down so that the $j$-th leaves of $f$ and $g$ are linked.

{\bf Fraction groups.}
{\it The} fraction group of $\cF$, denoted $\Frac(\cF,1)$, is the isotropy subgroup of $\Frac(\cF)$ associated to the object $1$. 
As a set $$\Frac(\cF,1)=\{ [t,s]:\ t,s \text{ trees with same number of leaves} \}.$$ 
Recall, that $\cF_\infty$ is the monoid of forests with roots in bijection with $\N$ that are all finitely supported.
Hence, it admits a fraction group $\Frac(\cF_\infty)$ when $\cF_\infty$ is a Ore monoid (which is equivalent in having that $\cF$ is a Ore monoid).
The two groups $\Frac(\cF,1)$ and $\Frac(\cF_\infty)$ embed in each other but are different in general, see \cite[Proposition 3.14]{Brothier22-FS}.

\begin{definition}
The group $\Frac(\cF,1)$ is called the {\it fraction group} of $\cF$.
We will often denote it by $G_\cF$ or $G$.
We may also say that $G_\cF$ is the {\it forest-skein group} of $\cF$ or in short the FS group of $\cF$.
A FS group is any group isomorphic to a certain $G_\cF$ as above.
\end{definition}

\begin{example}
The monochromatic free FS category $\cUF\la x\ra$ is a Ore category and its FS group is isomorphic to Thompson group $F$.
Note that $\cUF\la S\ra$ does not satisfy Ore's property unless $S$ is a singleton. Similarly, a monochromatic FS category is not left-cancellative unless it is free (i.e.~there are no skein relations).
\end{example}

We have the following convenient proposition concerning calculi of fractions.

\begin{proposition}
Let $M$ be a HP monoid, $\cF$ its associated FS category, and $\cF_\infty$ the associated FS monoid.
Let $Q$ be the property of being left-cancellative or of satisfying Ore's property and let $X$ being either $M,\cF,$ or $\cF_\infty$.
If $X$ has property $Q$, then so does all three $M,\cF,\cF_\infty$.
\end{proposition}

\begin{proof}
Fix $M,\cF,\cF_\infty$ as above and as usual denote by $\sigma$ the generating set of $M$.
Let us prove that if $M$ satisfies Ore's property, then so does $\cF$. We leave the proofs of the other assertions to the reader. 
Hence, assume that $M$ satisfies Ore's property. Using tensor product decomposition it is sufficient to show that given two trees $t,s$ (in the free forest category over $\sigma$) there exist forests $f,g$ satisfying $t\circ f\sim s\circ g$.
For each leaf $\ell$ of $t$ consider the longest geodesic path in $t$ ending in $\ell$ and made exclusively of left-edges. 
Now read the colouring of each of the vertices that this geodesic path goes through. This provides an element $u_\ell:=c_1\cdots c_k$ in the free monoid $\Mon\la\sigma\ra$ where $c_i$ is the colour of the $i$-th vertex of this geodesic path.
Similarly, define $v_\ell$ for the coloured tree $s$.
By Ore's property of $M$ there exists for each $\ell$ two words $p_\ell$ and $q_\ell$ in $\sigma$ so that $u_\ell\circ p_\ell$ and $v_\ell\circ q_\ell$ defines the same element in the monoid $M$.
Define now $f$ the forest with roots in bijection with the leaves of $t$ so that the $\ell$-th tree of $f$ is the left-vine coloured by $p_\ell$. Similarly, define $g$ using $q_\ell$.
By definition of the skein presentation of $\cF$ we deduce that $t\circ f\sim s\circ g$.
\end{proof}

\subsection{Colouring generating property}

Recall from \cite[Section 3.4]{Brothier22-FS} that a FS category $\cF$ has the {\it colouring generating property} (in short CGP) at a colour $a$ of $\cF$ if given any pair of trees $(t,s)$ with same number of leaves we can find a forest $f$ composable with $t$ so that both $t\circ f$ and $s\circ f$ have all vertices on their right branch coloured by $a$. If we have a calculus of fractions we then have that $[t,s]=[t',s']$ where $t'=t\circ f$ and $s'=s\circ f$.
For instance the following $t'$ is in the desired form:
\[\includegraphics{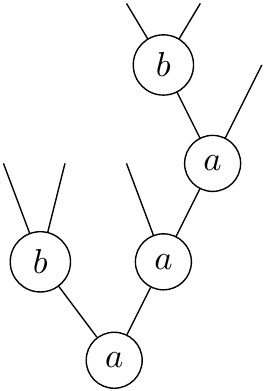}\]

The next proposition shows that all $\cF_M$ satisfy the CGP. This implies the remarkable fact that the fraction groups of $\cF_M$ and its FS monoid $(\cF_M)_\infty$ are isomorphic when they exist.

\begin{proposition}
Let $M$ be a HP monoid, $\cF$ its associated FS category, and $\cF_\infty$ the associated FS monoid.
If $M$ satisfies Ore's property, then $\cF$ has the CGP at any colour. 
In particular, if $M$ is a Ore monoid, then $\cF$ is a Ore category and the two fraction groups $\Frac(\cF,1)$ and $\Frac(\cF_\infty)$ are isomorphic.
\end{proposition}

\begin{proof}
Let $M$ be a HP monoid with generator set $\sigma=S$ satisfying the Ore property and let $\cF$ be its associated forest category.
Fix an element $a\in S$ (which is a colour of $\cF$ by definition).
Let $t$ be a tree of $\cF$.
Fix a representative $t_0$ of $t$ in the free forest category over $\cF$.
For each vertex $v$ on the right branch starting at the root of $t_0$ consider the left branch $b_v$ of $t_0$ starting from this vertex. 
Let $n_v$ be the natural number so that $b_v$ ends at the $n_v$-th leaf of $t_0$.
Let $w_v$ be the word in the alphabet $S$ obtained by reading all the vertices on the left branch $b_v$ (starting with the label of $v$).
We identify $w_v$ with an element of the monoid $M$.
Write $f\leq f'$ for forests when there exists $g$ so that $f'=f\circ g.$
Since $M$ satisfies Ore's property there exists $z\in M$ so that $a\leq z$ and $w_v\leq z$ (i.e.~$z$ is a common right-multiple of $a$ and $w_v$).
Hence, there exists $x,y\in M$ so that $ax=w_vy$.
Consider the tree $T_v$  equal to a left-vine coloured by the word $y$,
$$ \text{i.e.~if } y=s_1\cdots s_k \text{ with } s_i\in S, 1\leq i\leq k, \text{ then } T_v:= (s_1)_1 \circ (s_2)_1 \circ\cdots\circ (s_k)_1.$$
Attach now $T_v$ to the $n_v$-th leaf of $t_0$ and do that for each vertex $v$ on the right branch (starting from the root) of $t_0$.
We obtain a new tree $t_1$ satisfying $t_1\geq t_0$ and there exists a tree $\ti t_1$ equivalent to $t_1$ in $\cF$ whose right branch starting from the root has all its vertices coloured by $a$. 
Moreover, note the \textit{crucial fact} that $t_1=t_0\circ f$ where $f$ is a forest whose last tree is trivial, i.e.~$f=f'\ot I$ where $f'$ is a forest. 

Consider now a pair of trees $(t,s)$ in $\cF$ with the same number of leaves. 
Let $(t_0,s_0)$ be representatives of $(t,s)$ in the free forest category over $\cF$.
Construct $t_1$ as above so that $t_1=t_0\circ f$.
Now, perform a similar construction applied to $s_0\circ f$ obtaining a tree $s_0\circ f\circ g$ that is equivalent in $\cF$ to a tree having only $a$ on its right branch starting at the root.
Moreover, by construction we have that the tree $g$ decomposes as $g=g'\ot I$. 
This implies that the trees $t_1$ and $t_1\circ g$ have the same right branch (starting from the root).
We deduce that the pair $(t_0\circ f\circ g, s_0\circ f\circ g)$ is equivalent in $\cF\times \cF$ to a pair of trees having only $a$ on their right branch starting from the root, i.e.~$\cF$ has the CGP for $a$.
We have proven the first statement.

Assume now that $M$ is a Ore monoid. The previous proposition implies that $\cF$ is a Ore FS category. 
Now, using \cite[Section 3.4]{Brothier22-FS} we deduce that the two groups $\Frac(\cF,1)$ and $\Frac(\cF_\infty)$ are isomorphic.
\end{proof}

\subsection{Adding permutations and braids}
Consider a FS category $\cF$. Following \cite[Section 3.5]{Brothier22-FS} we can define three new categories $\cF^T, \cF^V$, and $\cF^{BV}.$
The symbols $F,T,V,BV$ stand for the usual three Richard Thompson groups and $BV$ for the Brin braided Thompson group.

Elements of $\cF$ are (classes of) coloured forests $(f,c)$ with $c:\Ver(f)\to S$ the vertex colouring.
Elements of $\cF^T$ are triples $(f,c,\pi)$ where $\pi$ is a cyclic permutation of the leaves of $f$ that we represent as a diagram in $\R^2$ on top of $f$ with $n$ strings going from $(\ell,0)$ to $(\pi(\ell),1)$ if all the leaves of $f$ have altitude $0$.
As usual we may remove the data of $c$ and simply write $(f,\pi)$.
Moreover, we think of $\pi$ as an element of $\cF^T$ so that $(f,\pi)=f\circ \pi$.

Similarly, we can define $\cF^V$ and $\cF^{BV}$. The first admits {\it all} permutations (rather than only cyclic ones) and the second braids. Hence, an element of $\cF^{BV}$ is of the form $f\circ \beta$ where $f$ is a forest with $n$ leaves and $\beta$ a braid with $n$ strands stacked on top of $f$.

When $\cF$ is a Ore category, then so are $\cF^T,\cF^V,$ and $\cF^{BV}$.
If $G:=\Frac(\cF,1)$ is the fraction group of $\cF$, then we obtain three additional groups that we denote by $G^T,G^V,$ and $G^{BV}$.
We may write $G^F$ for $G$.
We will often provide statement for the whole four groups of above. In that case we write $G^X$ and let $X$ be equal to either $F,T,V,$ or $BV$.
We say that $\cF^X$ and $G^X$ are the $X$-version of $\cF$ and $G$.

Note that we have canonical maps $\cF\into\cF^T\into\cF^V$ and $\cF^{BV}\onto\cF^V$.
They functorially extend into a chain of group embeddings $G\into G^T\into G^V$ and a quotient map $G^{BV}\onto G^V$.
In particular, if $\cF$ is the monochromatic free FS category, then $G^X\simeq X$ for $X=F,T,V,BV$.

We will often prove statement for $G$ that extends rather trivially to the other versions $T,V,$ and $BV$. In that case we will leave the proof of the extension to the reader.
If deducing the extension from the $F$-case is not trivial, then we will always include an additional argument.

\section{Shape-preserving skein presentation}\label{sec:shape}

We present FS categories admitting a skein presentation where only colours are changing (but not the underlying monochromatic trees). In that case we observe that there is a surjective morphism on the usual category of monochromatic forests. 
This permits to decompose into a semi-direct product the associated FS groups.
We do the whole analysis in the $F$-case for simplicity. We end the section by briefly explaining the general $X$-case with $X=T,V,$ or $BV$.

\subsection{Decolouring map and presentation}
Consider a nonempty set $S$ and the free FS category $\cUF\la S\ra$ over this set.
Let $\cUF\la x\ra$ be the monochromatic FS category with single colour $x$.

Given any forest $(f,c)$ of $\cUF\la S\ra$ with $c:\Ver(f)\to S$ we define 
$$D(f,c):=(f,\tilde c)\in \cUF\la x\ra$$
where $\tilde c:\Ver(f)\to \{x\}$ is the constant map.

Here is an example of a coloured tree $f$ and its decolouration:
\[\includegraphics{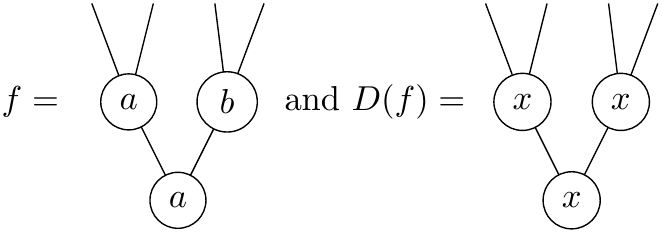}\]

\begin{definition}
The application $$D:\cUF\la S\ra\to \cUF\la x\ra$$ is called the \emph{decolouring} map or morphism.
It is a surjective morphism.
If $f,f'$ are two forests satisfying that $D(f)=D(f')$, then we say that $f$ and $f'$ have same {\it shape}.
\end{definition}

Consider now a skein presentation $(S,R)$ with associated category $\cF:=\FC\la S,R\ra.$ 

\begin{definition}
We say that $(S,R)$ is \emph{shape-preserving} when $D(t)=D(s)$ for all relation $(t,s)$ in $R$.
In that case we say that $\cF$ admits a shape-preserving skein presentation or is a shape-preserving FS category.
\end{definition}

Assume that $\cF$ admits a shape-preserving skein presentation.
Observe that all skein presentations of $\cF$ must be shape-preserving justifying the terminology.
Moreover, note that the morphism $D:\cUF\la S\ra\to\cUF\la x\ra$ factors through a (surjective) morphism 
$$D_\cF:\cF\to\cUF\la x\ra.$$
We will often write $D$ rather than $D_\cF$ in clear contexts.

\subsection{Colouring map}

Let $\cF$ be a FS category with colour set $S$.
Fix a colour $b\in S$ and a monochromatic forest $h\in \cUF\la x\ra$.
Define $$C_b^\cF(h)=(h,c_b)$$
where $c_b:\Ver(h)\to S$ is the constant map equal to $b$ everywhere.

\begin{definition}
The application $$C_b^\cF:\cUF\la x\ra\to \cF$$ is called the {\it $b$-colouring} map or morphism. 
\end{definition}
The map $C_b^\cF$ is indeed a morphism that is injective when $\cF$ is a Ore category, see \cite[Section 3.3]{Brothier22-FS}.

\begin{notation}
We will often suppress the superscript $\cF$ and will write $C_b$ for $C_b^\cF$.
Moreover, when $\cF$ is shape-preserving, then we may write $C_b$ for the composition $C_b\circ D$.
In that case, $C_b$ denotes an endomorphism of $\cF$ consisting of taking a coloured forest $(f,c)$ and recolouring all the vertices by $b$ that is 
$$(f,c)\mapsto (f,c_b) \text{ with } c_b:\Ver(f)\to S \text{ constant equal to } b.$$
\end{notation}

\subsection{Decomposition}

Assume $\cF:=\FC\la S|R\ra$ is a shape-preserving Ore FS category and fix a colour $a\in S$. 
Let $G:=\Frac(\cF,1)$ be the fraction group of $\cF$  and $F$ the one of $\cUF\la x\ra$ that we identify with the Thompson group.
The morphism $D:\cF\to\cUF\la x\ra$ extends uniquely (by functoriality) into a group morphism $G\to F$ via the formula:
$$[t,s]\mapsto [D(t),D(s)].$$
We denote this group morphism by $D$ as well. 
Similarly, $C_a$ defines a group morphism $F\to G$. 
We obtain the following decomposition of $G$.

\begin{proposition}
We have a short exact sequence:
$$1\to \ker(D)\to G\to F\to 1$$
which splits via the morphism $C_a:F\to G$.
Hence, the FS group $G$ is isomorphic to $\ker(D)\rtimes F$ for the action
$$g\cdot k:= C_a(g) k C_a(g)^{-1} \text{ for } g\in F, k\in \ker(D).$$
\end{proposition}

\begin{proof}
The map $D$ is surjective at the category level and thus extends into a surjective group morphism for the fraction groups.
Moreover, it has been proven that $C_a:F\to G$ is always injective for any FS group $G$, see \cite[Section 3.3]{Brothier22-FS}.
This implies that we have a short exact sequence as above.
It is obvious that the composition $D\circ C_a$ is the identity giving the split of the exact sequence.
The rest of the proposition follows easily.
\end{proof}

The decomposition
$$[t,s] = [t,C_a(t)] \circ [C_a(t),C_a(s)] \circ [C_a(s),s]$$
permits to appreciate the isomorphism $G\to \ker(D)\rtimes F$ where here $F$ is identified with pairs of trees uniquely coloured by $a$.

We will see in Section \ref{sec:decomposition} that when $\cF$ is constructed from a homogeneously presented monoid $M$, then we can further decompose $\ker(D)$ into an infinite direct sum $\oplus_{\Q_2} K_M$ indexed over the dyadic rationals of the unit interval $[0,1)$.
We may wonder what happens for arbitrary shape-preserving skein presentations and if $\ker(D)$ always decomposes in a nontrivial direct sum of groups. 
We observe below that a nontrivial decomposition always exists.
The proof is diagrammatic.

\begin{proposition}
The map
$$\theta:\ker(D)\times\ker(D)\to \ker(D), (g,h)\mapsto Y_a\circ (g\ot h)\circ Y_a^{-1}$$
is an isomorphism of groups.
\end{proposition}
\begin{proof}
We first explain what the map $\theta$ is using fractions.
Consider $g,h\in\ker(D)$. 
There exists two pairs of trees $(t,t')$ and $(s,s')$ so that $g=[t,t']$ and $h=[s,s']$.
Moreover, note that $t$ and $t'$ (resp.~$s$ and $s'$) have same shape.
Now, $\theta(g,h)$ is the fraction $$[Y_a\circ(t\ot s), Y_a\circ (t'\ot s')]$$ where recall $t\ot s$ is the forest with two roots whose first tree is $t$ and second $s$.
It is obvious that $Y_a\circ(t\ot s)$ and $Y_a\circ (t'\ot s')$ have same shape implying that $\theta$ is valued in $\ker(D)$.
By definition of the calculus of fractions we have that $\theta$ is indeed a group morphism.
If $\theta(g,h)=e$, then $Y_a\circ(t\ot s)=Y_a\circ(t'\ot s')$.
Since $\cF$ is left-cancellative we deduce that $t\ot s=t'\ot s'$ meaning that $t=t'$ and $s=s'$ that is $(g,h)=(e,e).$
Hence, $\theta$ is injective.\\
Let us prove that $\theta$ is surjective which is slightly surprising (indeed if we extend $\theta$ into a map from $G\times G$ to $G$ in the obvious way, then it is no longer surjective).
Consider $k=[u,v]\in\ker(D)$. 
By Ore's property applied to $Y_a$ and $u$ there exist some forests $f,f'$ so that $Y_a\circ f = u\circ f'$.
Hence, $k=[u\circ f', v\circ f']=[Y_a\circ f, v\circ f']$.
By using again Ore's property applied to $Y_a$ and $v\circ f'$ we find $p,p'$ forests so that $Y_a\circ p = v\circ f'\circ p'$.
We deduce that $k=[Y_a\circ f\circ p', Y_a\circ p].$
Now, $f\circ p'$ decomposes as a tensor product of two trees $t\ot s$ and likewise $p=t'\ot s'.$
Since $k\in\ker(D)$ we necessarily have that $D(t\ot s)=D(t'\ot s')$ and thus $D(t)=D(t')$ and $D(s)=D(s')$.
Therefore, $t$ has same number of leaves than $t'$ (which is the crucial part of the argument) and $[t,t']\in\ker(D)$.
We deduce that $k=\theta(g,h)$ where $g=[t,t']$ and $h=[s,s'].$
\end{proof}

By slightly adapting the proof of above we obtain that for any tree $t$ with $n$ leaves we have an isomorphism 
$$\theta_t:\ker(D)^n\to \ker(D), (g_1,\cdots,g_n)\mapsto t\circ (g_1,\cdots,g_n)\circ t^{-1}.$$

\subsection{Extension to the other Thompson's groups}
The whole discussion extends easily to the three other cases: $X=T,V,BV$.
Indeed, observe that the maps $D$ and $C$ extend (functorially) to the $X$-versions as:
$$D^X:\cF^X\onto\cUF\la x\ra^X \text{ and } C_b^X:\cF\la x\ra^X\into \cF^X.$$
The first is surjective while the second is injective.
Observe now that the fraction group of $\cUF\la x\ra^X$ is isomorphic to $X$.
Moreover, note that $\ker(D^X)=\ker(D)$.
We deduce a new split exact sequence
$$1\to \ker(D)\to G^X \to X\to 1.$$
In particular, $G^X$ is isomorphic to $\ker(D)\rtimes X$ for $X=F,T,V,BV$.

\section{Forest-skein groups associated to homogeneously presented monoids}\label{sec:decomposition}
The aim of this section is to provide an explicit description of $\Frac(\cF_M,1)$: the fraction group of a FS category built from a monoid $M$. 
The description will be done using wreath products. 
It will permit to show that various properties satisfied by $\Frac(M)$ passes through the much larger group $\Frac(\cF_M,1)$. Moreover, we will be able to fully classify this specific class of FS groups in terms of $\Frac(M)$.
Note that all FS categories $\cF_M$ are in particular shape-preserving. 

In all this section $M$ is a Ore HP monoid with presentation $\pi=(\sigma,\rho)$ and associated fraction group $\Frac(M)$.

\subsection{Decomposition of $\Frac(M)$}
 
Using the length map we are going to decompose $\Frac(M)$ into a semi-direct product. 
The kernel of the length map at the group level will be a key component in the description of the FS group $\Frac(\cF_M,1)$.
Recall that the length map $\lambda:M\to\N$ extends into a surjective group morphism 
$$\lambda:\Frac(M)\to\Z, [u,v]\mapsto [\lambda(u),\lambda(v)].$$
Here, $[\lambda(u),\lambda(v)]=\lambda(u)-\lambda(v)$ via the usual identification of $\Frac(\N)$ and $\Z$.
Take any letter $a\in\sigma$ and note that 
$$\Z\to \Frac(M), \ 1\mapsto a$$
is a right-inverse of $\lambda$.
We deduce the following decomposition.

\begin{proposition}
Denote by $K_M$ the kernel of $\lambda:\Frac(M)\to\Z$ and fix a letter $a\in\sigma$.
We have a split exact sequence 
$$1\to K_M\to \Frac(M)\to \Z\to 1$$
giving a group isomorphism
$$\Frac(M)\simeq K_M\rtimes \Z$$
where the action $\Z\act K_M$ is given by the formula:
$$n\cdot k:= a^n k a^{-n} \text{ for } k\in K_M \text{ and } n\in\Z.$$
\end{proposition}

\begin{example}
\begin{enumerate}
\item If $M=\Mon\la a\ra=\N$ is the free monoid in one generator, then $\Frac(M)=\Z$ and $K_M$ is the trivial group.
\item More generally, consider $M=\N^\sigma$ the free abelian monoid over a nonempty set $\sigma$ equipped with its classical presentations 
$$\Mon\la \sigma| ab=ba, a,b\in \sigma\ra.$$
Its fraction group is $\Z^\sigma$ and decomposes as a direct product $K_M\times \Z$ where $a\in \sigma$ is a fixed generator, $\Z$ is generated by $a$, and $K_M=\Z^{\sigma\setminus\{a\}}$.
\item Consider $$M=\Mon\langle a,b| a^2=b^2,ab=ba\rangle.$$
This is a Ore HP monoid with fraction group $G$.
Note that this is an example of a thin but non-Gaussian monoid in the sense of Dehornoy, see \cite{Dehornoy02}.
We have that $K_M=\{e,ab^{-1}\}\simeq \Z/2\Z$.
We obtain that $G\simeq \Z/2\Z\times\Z$ with $\Z/2\Z$ corresponding to $K_M$ and $\Z$ to the cyclic subgroup generated by $a$.
\item
Consider the Thompson monoid with its classical infinite presentation:
$$F^+:=\Mon\la x_n: n\geq 1 | x_q x_j = x_j x_{q+1}, 1\leq j<q\ra.$$
This is a Ore HP Monoid.
Note that $F^+$ is equal to the forest monoid of the monochromatic forest category and $x_n$ corresponds to the elementary forest having a single caret at the $n$-th root.
If we think of $F$ as the group of classes of pairs of binary trees, then we have the following group isomorphism:
$$\Frac(F^+)\to F, x_n\mapsto [t_{n+1} \circ x_n, t_{n+2}]$$
where $t_{n+1}=x_{1,1}x_{2,2}\cdots x_{n,n}$ is the tree with $n+1$ leaves with a long right-branch, i.e.~a right-vine with $n$ carets.
After identifying $\Frac(F^+)$ with $F$ under this isomorphism we find that the kernel $K$ of the length group morphism $F\to \Z$ is equal to the set of all $g\in F$ for which there exists a pair of trees $(t,s)$ so that $g=[t,s]$ and the length between the root and the right-most leaf of $t$ is equal to the length between the root and the right-most leaf of $s$.
If we identify $F$ in the usual way as a subgroup of homeomorphisms of the unit interval, then this means that the slope of $g$ at $1$ is equal to $1.$
Therefore, $$K=\{ g\in F:\ g'(1)=1\}.$$
Now, a section of $F\to \Z$ is given by the group morphism
$$\Z\to F, 1\mapsto x_1.$$
Hence, $F\simeq K\rtimes \Z$ where $\Z\act K$ is given by conjugation by the classical first generator $x_1$ of Thompson's group $F$.
\item 
For each $n\geq 2$ we consider the Artin braid monoid over $n$ strands with its classical presentation
$$B_n^+:=\Mon\la g_1,\cdots,g_n| g_k g_{k+1} g_k = g_{k+1} g_k g_{k+1} \ , \ g_ig_j=g_j g_i\ra$$
where $$1\leq k\leq n-2, 1\leq i,j\leq n-1, \text{ and } |i-j|\geq 2.$$
It is a Ore HP monoid with fraction group $B_n=\Frac(B_n^+)$ the Artin braid group over $n$ strands. 
The kernel $K$ of the length morphism is the subgroup of braids that have equal number of under and over crossings.
Consider the copy of $\Z$ inside $B_n$ that is generated by $g_1.$
We have $B_n\simeq K\rtimes \Z$ where $\Z\act K$ consists in conjugating $K$ by $g_1$, i.e.~taking a braid of $K$ and adding on top of it one overcrossing of the first two strands and one undercrossing on the bottom for the same two first strands.
\end{enumerate}
\end{example}

\subsection{Decomposition of forest-skein groups}\label{sec:decomposition}

We recall some notations before stating the structural theorem describing $\Frac(\cF,1).$

Let $\Q_2=[0,1)\cap \Z[1/2]$ be the dyadic rationals in $[0,1)$ (sometime identified with the dyadic rationals in the unit torus or with $\Z[1/2]/\Z$).
The classical action $F\act [0,1)$ restricts into $F\act \Q_2$.
We may as well identify $\Q_2$ with the finitely supported sequences of the Cantor space $\fC:=\{0,1\}^\N$ via the map $$\fC\onto [0,1], (x_n)_n\mapsto \sum_n \frac{x_n}{2^n}.$$
The usual action $V\act \fC$ (consisting in changing finite prefixes) restricts to an action $V\act \Q_2$. 
Finally, recall that there is a canonical map $BV\onto V$ from the braided Thompson group onto $V$. This maps consists in changing over and under crossings into a single crossing.
This produces an action $BV\act \Q_2$ obtained from the composition $BV\onto V\to \Aut(\Q_2)$.

\begin{definition}
If $A,B$ are groups, $B\act Q$ is an action on a set $Q$, and $\tau:B\times Q\to \Aut(A)$ is a cocycle valued in the automorphism group of $A$, then we write 
$$A\wr_{Q,\tau}B=\oplus_Q A\rtimes  B$$ for the \emph{twisted permutational restricted wreath product}.
Here, $\oplus_Q A$ denotes the set of finitely supported maps $h:Q\to A$ and $B\act \oplus_Q A$ is the generalised twisted shift-action described by the formula:
$$b\cdot h(q):= \tau(b,q)(h(b^{-1}q)), h\in \oplus_Q A, b\in B, q\in Q.$$
The map $\tau:B\times Q\to\Aut(A)$ is called the \emph{twist}. If $\tau$ is trivial, then we say that the wreath product is {\it untwisted}.
We may suppress the subscripts $Q$ and $\tau$ in clear contexts.
\end{definition}

We are now ready to state and prove our structural theorem for FS groups constructed from HP monoids.

\begin{theorem}\label{theo:FCWP}
Consider a Ore HP monoid $M=\Mon\la \sigma|\rho\ra$, a generator $a\in \sigma$, the associated FS category $\cF$ and FS groups $G^X$ with $X=F,T,V,BV$.
Recall that $K_M$ denotes the kernel of the length-map $\Frac(M)\onto \Z$.

There is a group isomorphism
$$\Frac(\cF,1)\simeq K_M\wr_{\Q_2,\tau}X$$
where $K_M\wr_{\Q_2,\tau}X$ is the twisted wreath product described by the formula:
$$(h\cdot k)(q) = a^{-\log_2(h'(h^{-1} q))}\cdot k(h^{-1} q) \cdot a^{\log_2(h'(h^{-1} q))} \text{ for all } h\in F, q\in \Q_2, k\in K_M.$$

In particular, if the generator $a\in\sigma$ is central in $\Frac(M)$, then the twist is trivial.
\end{theorem}

\begin{proof}
We treat the $F$-case. The other cases easily follow as we will briefly explain at the end of the proof.
We start by explaining how to associate a dyadic rational to the leaf of a tree.
Recall that the vertex set $\Ver(t_\infty)$ of the infinite rooted regular binary tree $t_\infty$ is in bijection with the set $SDI$ of standard dyadic intervals of $[0,1]$ via a map $\nu\mapsto I_\nu$.
The root is sent to $[0,1]$, and if $\nu_0,\nu_1$ are the left and right immediate children of a vertex $\nu$, then $I_{\nu_0}$ and $I_{\nu_1}$ are the first and second half of $I_\nu$.
Now, consider the map $$S:\Ver(t_\infty)\to\Q_2, \nu\mapsto \min(I_\nu)$$
that sends a vertex $\nu$ of $t_\infty$ to the first point of the sdi associated to $I_\nu$. 
In particular, note that $S(\nu_0)=S(\nu)$ and $S(\nu_1)$ is the midpoint of $I_\nu$ for all vertex $\nu$.
The map $S$ is surjective and two vertices are sent to the same dyadic rational if and only if the geodesic path between them is only made of left-edges.

We will now construct an isomorphism from $\ker(D)$ to $\oplus_{\Q_2}K_M$.
Consider $g=[x,y]\in \ker(D)$ where we fix a representative $(x,y)$.
Hence, $x,y$ are trees with vertices coloured by elements of the generator set $\sigma$ of the monoid $M$.
Let $\ell$ be a leaf of $x$ and consider the longest geodesic path in $x$ made of only left-edges and ending at $\ell$.
We call it the {\it maximal left-geodesic} of $\ell$, denoted $mlg(x,\ell)$ or simply $mlg(\ell)$, and say that the first vertex of this path is its {\it origin}.
Now, let us read the decoration of this path which gives us a word $m_\ell$ in the alphabet $\sigma$. 
The order is such that the right-most letter of $m_\ell$ corresponds to the colouring of the closest vertex to the leaf $\ell$ in $x$.
Do the same for the leaf of $y$ corresponding to $\ell$ obtaining a word $n_\ell$.
The leaf $\ell$ corresponds to a unique vertex of $t_\infty$ and thus to a dyadic rational $S(\ell)$.

Define now $R_\ell$ to be the number of edges between the root of the tree $t$ and the origin of the maximal left-geodesic of $\ell$ inside $x$. 
In other words, 
$$R_\ell=d_x(\ell,root)-|mlg(x,\ell)|$$ 
where $d_x(\ell,root)$ is the tree-distance between $\ell$ and the root of $x$ and where $|mlg(x,\ell)|$ is the length of $mlg(x,\ell)$.

Define the map
$$k_{x,y}: \Q_2\to \Frac(M),\ S(\ell)\mapsto a^{R_\ell} \cdot m_\ell\cdot n_\ell^{-1} \cdot a^{-R_\ell}$$
supported on the finite set 
$$\{S(\ell):\ \ell\in \Leaf(x)\}.$$

Here is an example with
\[\includegraphics{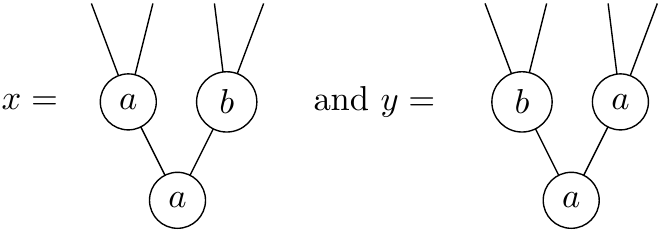}\]

The lists of words $m_\ell$ and $n_\ell$ and the list of values $R_\ell$ (where leaves are ordered from left to right) are:
$$(aa,e,b,e), (ab,e,a,e) \text{ and } (0,2,1,2).$$

We deduce that $k_{x,y}$ is the map supported in $\{0,1/4,1/2,3/4\}$ taking the following values:
$$k_{x,y}(0)=aa (ab)^{-1}, k_{x,y}(1/4)=e, k_{x,y}(1/2)=a^1\cdot ba^{-1}\cdot a^{-1}=aba^{-2}, \text{ and } k_{x,y}(3/4)=e.$$

By assumption, $D(g)=e$ meaning that $D(x)=D(y)$ and thus the underlying uncoloured tree of $x$ is equal to the one of $y$.
This implies that $m_\ell$ and $n_\ell$ have the same length and thus 
$$m_\ell \circ n_\ell^{-1}\in\ker(\Frac(M)\onto\Z)=K_M.$$
Since $K_M$ is closed under conjugation by $a$ we deduce that 
$k_{x,y}\in \oplus_{\Q_2}K_M$.
It is not hard to see that $k_{x,y}$ only depends on $g$ and note on the choice of the representative $(x,y).$
For instance, see that if we add a caret of colour $b$ to $x$ and $y$ at the leaf $\ell$, then we obtain two new leaves $\ell_0,\ell_1$ that are immediate left and right children of $\ell$.
Moreover, $S(\ell_0)=S(\ell)$, $R_{\ell_0}=R_\ell$ and
$$m_{\ell_0} n_{\ell_0}^{-1} = (m_\ell \cdot b ) (n_\ell \cdot b)^{-1} = m_\ell n_\ell^{-1}.$$ 
Further, $m_{\ell_1}=n_{\ell_1}=e$ because the mlg of $\ell_1$ is trivial.
Therefore, if $x',y'$ are the trees obtained by adding the $b$-caret we obtain that $k_{x,y}=k_{x',y'}.$
We now write $k_g$ instead of $k_{x,y}$.
We have defined a map
$$k:\ker(D)\to\oplus_{\Q_2}K_M, g\mapsto k_g.$$

{\bf Claim: The map $k$ is a group isomorphism.}

We have shown that $k$ is well-defined.
Consider $g=[x,y]$ and $h=[w,z]$ in $\ker(D).$
Using Ore's property we may assume that $y=w$ up to growing our tree-representatives.
Note that $D(x)$ is equal to $D(y)$ and $D(z)$ since $D(g)=D(h)=e$.
For each leaf $\ell$ of $x$ we write $m^x_\ell$ the element of $M$ obtained by reading colours on $x$ as explained above and similarly define $m^y_\ell$ and $m^z_\ell$.
Similarly, we define $R^x_\ell$ and note that $R^x_\ell=R^y_\ell=R^z_\ell$ since $R^x_\ell$ only depends on the shape of $x$.
We have that $k_g$ and $k_h$ are supported on $\{S(\ell):\ \ell\in\Leaf(t)\}$.
Moreover, note that 
\begin{align*}
k_g(S(\ell)) \circ k_h(S(\ell)) & = [a^{R^x_\ell} \cdot m^x_\ell (m^y_\ell)^{-1} a^{-R^x_\ell}] \circ [ a^{R_\ell^y} \cdot m_\ell^y (m_\ell^z)^{-1}\cdot  a^{-R_\ell^y}] \\
& = a^{R_\ell^x} \cdot m^x_\ell \circ (m^z_\ell)^{-1} \cdot a^{-R_\ell^x} = k_{g\circ h}(S(\ell))
\end{align*}
since $g\circ h = [x,z].$
This proves that $k$ is a group morphism.

Consider $r\in\Q_2, p\in K_M$ and the element $f_{r,p}:\Q_2\to K_M$ supported at $r$ taking the value $p$.
We are going to construct $[x,y]\in \ker(D)$ satisfying $k_{[x,y]}=f_{r,p}.$

Define the map $R:\Q_2\to \N$ as follows. First, identify $\Q_2$ with all the finitely supported binary sequence in $0,1$ using the map 
$$(q_1,q_2,\cdots)\mapsto \sum_{c=1}^\infty \frac{q_c}{2^j}.$$
Given a finitely supported sequence $q=(q_1,q_2,\cdots )$ we write $R(q)$ for the the largest $c$ so that $q_c=1$.
Note that if $\ell$ is the leaf of a tree $t$ and that we identify $\ell$ with a binary sequence $q_\ell$ of $0,1$ in the usual way (with $0$ standing for left-edges and $1$ for right-edges), then $R(q_\ell)=R^t_\ell$ where $R^t_\ell$ is the natural number previously used to define the map $k$.

We consider $R(r)$ and the element $$\hat p:= a^{-R(r)}\cdot p \cdot a^{R(r)}.$$ 
We have that $\hat p$ belongs to $K_M$.
Therefore, there exists a pair of elements $(m,n)$ in the monoid $M$ such that $\hat p = m\circ n^{-1}$ and such that $m$ and $n$ have same word-length equal to a certain $L.$
Since $S:\Ver(t_\infty)\to \Q_2$ is surjective there exists $\nu$ so that $S(\nu)=r$.
Up to taking a left-child of $\nu$ we may assume that the maximal left-geodesic (mlg) of $\nu$ is longer than $L$.
Consider now an uncoloured tree $t$ so that $\nu$ is one of its leaf.
Colour the maximal left-geodesic of $\nu$ inside $t$ by the word $m \cdot a^k$ where $k=L-|mlg(t,\nu)|\geq 0$.
Colour all other vertices of $t$ with $a$. We obtain a coloured tree $x$.
Similarly colour $t$ into a coloured tree $y$ so that the mlg of $\nu$ inside $t$ is coloured by $n\cdot a^k$ and all other vertices by $a$.
We obtain that if $g=[x,y]$, then 
$$k_g(S(\nu)) = k_{x,y}(S(\nu)) = a^{R(r)}\cdot m a^k (n a^k)^{-1} a^{-R(r)} = a^{R(r)}\cdot m n^{-1} a^{-R(r)}=p$$
and $k_g$ takes trivial values at any other dyadic rationals.
Therefore, $k_g=f_{r,p}$ and since $\{f_{r,p}:\ r\in\Q_2, p\in K_M\}$ generates the group $\oplus_{\Q_2}K_M$ we have proven that $k$ is surjective.

It remains to show that $k$ is injective.
Assume that $g=[x,y]\in \ker(D)$ is in the kernel of $k$.
In particular, $k_g(S(\ell))=e$ for all leaf of $x$, i.e.~$m^x_\ell=m^y_\ell$.
This implies that $x$ and $y$ have the same vertex-colouring (modulo the FS relations of $\cF$).
Since $x$ and $y$ have same shape we deduce that $x=y$ and thus $g=e$.
This proves the claim.

We now analyse the action $F\act\oplus_{\Q_2}K_M$ obtained by pulling back the action $F\act\ker(D)$ via the isomorphism $k:\ker(D)\to \oplus_{\Q_2}K_M$.
Recall that we have fixed a colour $a\in\sigma$ and we have considered the embedding 
$$C_a:F\to \Frac(\cF,1)$$
where $C_a$ is the colouring map.
Consider $g=[x,y]\in \ker(D)$ and $h\in C_a(F)$ so that $h=[C_a(t),C_a(s)]$ where $t,s$ are two monochromatic trees.
We want to compare $k_g$ with $k_{hgh^{-1}}$.
The Ore property assures the existence of two forests $u,v$ satisfying that 
$$[x,y]=[C_a(s)\circ u, C_a(s)\circ v].$$
Observe now that 
$$hgh^{-1} = [C_a(t)\circ u, C_a(t)\circ v].$$

{\bf Notations.} 
Let $\ell$ be the $i$-th leaf of $C_a(s)\circ u$.
Note that $u$ is a forest that is stacked on top of the tree $C_a(s)$.
Hence, we can express $u$ as a list of trees $(u_1,\cdots,u_d)$ where $d$ is the number of leaves of $C_a(s)$.
Let $1\leq j\leq d$ be the natural number satisfying that the leaf $\ell$ belongs to $u_j$.
We write $r$ for the $j$-th leaf of $C_a(s)$ and $\ti r$ for the $j$-th leaf of $C_a(t)$.

We have that 
$$k_g(S(\ell))=a^{R_\ell}\cdot m_\ell\cdot n_\ell^{-1}\cdot a^{-R_\ell}$$ as previously explained. 
We put $$m_i:=m_\ell,n_i:=n_\ell,R_i:=R_\ell$$ to express that we consider the $i$-th leaves of the tree $C_a(s)\circ u$.
Consider now $\ti\ell$ the $i$-th leaf of $C_a(t)\circ u$.
Observe that $$S(\ti\ell)=h\cdot S(\ell).$$
Let us compute $k_{hgh^{-1}}(S(\ti\ell))$ in terms of $k_g(S(\ell))$ and $h$.
Using similar notations than above we have that 
$$k_{hgh^{-1}}(S(\ti\ell))=a^{\ti R_i}\cdot \ti m_i\cdot \ti n_i^{-1}a^{-\ti R_i}.$$

{\bf Case 1.}
Assume that the mlg of the $i$-th leaf $\ell$ of $C_a(s)\circ u$ is strictly contained inside $u$. By this we mean that this mlg does not pass through a leaf of $C_a(s)$.
In that case, $m_i$ only depends on $u$ and thus $m_i=\ti m_i$.
Similarly, $n_i=\ti n_i$ since $C_a(s)\circ u$ and $C_a(s)\circ v$ have same shape.
We deduce that $$k_{hgh^{-1}}(h\cdot S(\ell)) = a^{\ti R_i - R_i}\cdot k_g(S(\ell)) \cdot a^{-\ti R_i + R_i}.$$
Note that $R_i=d_s(j,root) + Q$ where $d_s(j,root)$ is the distance between the $j$-th leaf of $s$ and its root and where $Q$ is the distance between the origin of the mlg of $\ell$ and the root $r$ of $u$.
We have that $\ti R_i = d_t(j,root) + Q$.
Therefore, $R_i-\ti R_i = d_s(j,root) - d_t(j,root).$
We now analyse this quantity using the slope of $h$ at the point $S(\ell)$.
Indeed, by definition of the action of $F$ on the real unit interval $[0,1]$ we have that $h$ sends the interval $I_{s,j}$ to $I_{t,j}$ in the unique order-preserving affine way where $I_{s,j}$ stands for the standard dyadic interval associated to the $j$-th leaf of the tree $s$.
The Lebesgue measure of $I_{s,j}$ is equal to $2^{-d_s(j,root))}$.
Therefore, the slope of $h$ at the point $S(\ell)$ is equal to 
$$h'(S(\ell))=2^{d_s(j,root)-d_t(j,root)}=2^{R_i-\ti R_i}.$$
We deduce that 
$$k_{hgh^{-1}}(h\cdot S(\ell)) = a^{-\log_2(h'(S(\ell)))} \cdot k_g(S(\ell)) \cdot a^{\log_2(h'(S(\ell)))}.$$

{\bf Case 2.}
Assume now that the mlg of the $i$-th leaf $\ell$ of $C_a(s)\circ u$ is not strictly contained inside $u$.
This path goes through a root $r$ of the forest $u$ and stops at a vertex $\nu$ that is either equal to $r$ or is an internal vertex of the tree $t$.
As above, we say that $r$ is the $j$-th leaf of $s$ (identified with the $j$-th root of $u$).
Note that $S(\ell)=S(r)$.
Moreover, $R_i= d_{s}(\nu,root)$ and $\ti R_i = d_t(\ti\nu,root)$ using obvious notations.
Let us compute the word $\ti m_i$.
Since $C_a(s)$ only contains $a$-coloured vertices we deduce that 
$$m_i = a^{d_s(j,\nu)}\cdot m_i(u)$$ where 
$m_i(u)$ is the word obtained from reading the colours of vertices in the mlg of the $i$-th leaf of $u$ inside $u$.
Similarly, $$n_i=a^{d_s(j,\nu)} \cdot m_i(v) , \ \ti m_i=a^{d_t(j,\ti \nu)}\cdot m_i(u) , \text{ and } \ti n_i = a^{d_t(j,\ti \nu)} \cdot m_i(v).$$
All together we deduce that 
\begin{align*}
k_{hgh^{-1}} (S(\ti\ell)) &= a^{\ti R_i}\cdot \ti m_i\cdot \ti n_i^{-1}\cdot a^{-\ti R_i }\\
& = a^{d_t(\ti\nu,root)} \cdot \ti m_i\cdot \ti n_i^{-1} \cdot a^{-d_t(\ti\nu,root)}\\
& = a^{d_t(\ti\nu,root)} \cdot a^{d_t(j,\ti \nu)}\cdot m_i(u) \cdot m_i(v)^{-1}\cdot a^{-d_t(j,\ti \nu)}\cdot a^{-d_t(\ti\nu,root)}\\
& = a^{d_t(j,root)} \cdot m_i(u) \cdot m_i(v)^{-1} \cdot \cdot a^{-d_t(j,root)}\\
& = a^{d_t(j,root)} \cdot a^{-d_s(j,\nu)} \cdot m_i\cdot n_i^{-1} \cdot a^{d_s(j,\nu)}\cdot a^{-d_t(j,root)}\\
& = a^{d_t(j,root)} \cdot a^{-d_s(j,\nu)} \cdot a^{-d_s(\nu,root)} \cdot k_g(S(\ell)) \cdot a^{d_s(\nu,root)} \cdot a^{d_s(j,\nu)}\cdot a^{-d_t(j,root)}\\
& = a^{d_t(j,root)-d_s(j,root)} \cdot k_g(S(\ell)) \cdot a^{d_s(j,root)-d_t(j,root)}\\
& = a^{-\log_2(h'(S(\ell)))}\cdot k_g(S(\ell)) \cdot a^{\log_2(h'(S(\ell)))}.
\end{align*}

In conclusion, we have that $$\supp(k_{hgh^{-1}})=h\cdot \supp(k_g)=\{S(\ti\ell):\ \ti\ell \text{ a leaf of } C_a(t)\circ u\}.$$
If $q_i$ is the dyadic rational corresponding to the $i$-th leaf of $C_a(t)\circ u$, then we have the formula:
$$k_{hgh^{-1}}(q_i) = a^{-\log_2(h'(h^{-1}q_i))}\cdot k_g(h^{-1}q_i)\cdot a^{\log_2(h'(h^{-1}q_i))}.$$

This proves the theorem in the $F$-case.
We leave the $T,V$, and $BV$ cases to the readers. 
The $T$ and $V$-cases are rather identical. The only difference being that given a fixed $i$-th leaf of $C_a(s)\circ u$ we consider the $\sigma(i)$-th leaf (rather than the $i$-th leaf) of $C_a(t)\circ u$ where $\sigma$ is a (possibly nontrivial) permutation.
The remaining $BV$-case follows easily from the $V$-case using that the action $BV\act \Q_2$ is the composition of the quotient map $BV\onto V$ and the action $V\act \Q_2.$
\end{proof}

\begin{definition}\label{def:untwisted}
Consider a Ore HP monoid $M=\Mon\la \sigma|\rho\ra$, its fraction group $\Frac(M)$, and its associated FS groups $G,G^T,G^V,G^{BV}$.
If there exists a generator $a\in \sigma$ of $M$ that is central in the group $\Frac(M)$, then we say that we are in the untwisted case and say that the FS groups $G^X$ are untwisted for $X=F,T,V,BV.$
This is equivalent to say that the wreath products $K_M\wr_{\Q_2,\tau}X$ constructed in the proof of Theorem \ref{theo:FCWP} has a trivial twist.
\end{definition}

\begin{remark}\label{rem:twist}
In \cite{Brothier22,Brothier19WP} we have considered pairs $(\Gamma,\alpha)$ where $\Gamma$ is a group and $\alpha$ is an automorphism of $\Gamma$.
From $(\Gamma,\alpha)$ we have built a category $\mathcal C(\Gamma,\alpha)$ consisting of monochromatic binary forests with leaves decorated by elements of $\Gamma$. 
The composition of forests with tuples of elements of $\Gamma$ is described by the following single formula: 
$$g\circ Y= Y\circ (\alpha(g),e)$$ 
where $g\in \Gamma$ and where $Y$ is the (unique) monochromatic tree with two leaves.
This latter category has a calculus of fractions and provides the fraction groups $H,H^T,H^V,H^{BV}$.
We have shown in \cite{Brothier22} that $H^V$ is isomorphic to the twisted (permutational restricted) wreath product 
$\oplus_{\Q_2}\Ga\rtimes V$
with the twist expressed by the formula:
$$(hkh^{-1})(q) = \alpha^{\log_2(h'(h^{-1}q)}(k(h^{-1} q)) ,\ h\in V, k\in\oplus_{\Q_2}\Gamma, q\in \Q_2.$$
We get identical descriptions for the $F$ and $T$-cases by restrictions. 
For the $BV$-case we change the formula in the obvious way using the quotient map $BV\onto V$.

We note that if $(\Gamma,\alpha)=(K_M,\ad(a^{-1}))$, then the group $G^X$ of this present article is isomorphic to the groups $H^X$ for $X=F,T,V,BV$.
Using our previous results from \cite{Brothier22,Brothier19WP} we then immediately deduce the statements of this present article on the Haagerup property (for the $F,T,V$-cases) and on the classification of the groups $G^V$.
\end{remark}

\section{Consequences of the structural theorem of forest-skein groups}\label{sec:corollaries}

We have proven that $\Frac(\cF,1)$ is isomorphic to the twisted permutational restricted wreath product $K_M\wr_{\Q_2,\tau}F$ associated to $F\act \Q_2$ and $K_M$.
In particular, $\Frac(\cF,1)$ only depends on $K_M$ and on the twist $\tau$ (this latter being deduced from a section of $\Frac(M)\onto\Z$).
These groups previously appeared in three independent contexts.
Tanushevski constructed them using categories, so did Witzel and Zaremsky using pure cloning systems, and finally the author using (monoidal and covariant) functors from the free monochromatic FS category $\cUF\la x\ra$ into the category of groups \cite{Tanushevski16, Tanushevski17, Witzel-Zaremsky18, Brothier22}. 
The interested reader can consult the appendix of \cite{Brothier21} for a comparison of these three approaches.

\subsection{Classification up to isomorphisms}

We previously completely classified a large class of twisted wreath products of the form $\Ga\wr_{\Q_2}V$ but when the largest Thompson group $V$ is acting. 
This classification permits us to distinguish $V$-versions of FS groups using the kernel $K_M$.
Here we investigate the $F$ and $T$-case. We provide two independent and elementary proofs which extend the $V$-case result. Each of the two proofs is based on a trick that consists in finding a certain characteristic subgroup that clearly remembers the group $\Ga.$ 
The proofs work for all possible twists.
Note that we do not treat the $BV$-case as we explained in the introduction.

\begin{theorem}\label{theo:classifWP}
Consider $\Q_2=\Z[1/2]\cap [0,1)$ and let $X$ be any of the three Thompson groups $F,T,$ or $V$.
Consider the usual action $X\act \Q_2$ and for each group $\Ga$ the associated (possibly twisted) wreath product $\Ga\wr_{\Q_2}X:=\oplus_{\Q_2}\Ga\rtimes X$.

If $\Ga,\ti\Ga$ are groups satisfying that $\Ga\wr_{\Q_2}X$ and $\ti\Ga\wr_{\Q_2}X$ are isomorphic, then $\Ga$ and $\ti\Ga$ are isomorphic.
\end{theorem}

\begin{proof}
Fix a group $\Ga.$
The case $X=V$ has already been proven in \cite{Brothier22}.
We are going to prove the theorem for $X=F$ and $X=T$ separately. 

We start by considering $X=F$.
Observe that $F\act \Q_2$ has two orbits: the singleton $\{0\}$ and its complement $U:=\Q_2\setminus\{0\}=\Z[1/2]\cap (0,1).$
We deduce that the group $\Ga\wr_{\Q_2}F$ decomposes into a direct product $\Ga_0 \times (\Ga\wr_U F)$ where $\Ga_0$ denotes the copy of $\Ga$ associated to the point $0\in\Q_2.$
We are going to show that $\Ga_0$ is a characteristic subgroup of $\Ga\wr_{\Q_2}F$, i.e.~every automorphism of $\Ga\wr_{\Q_2}F$ restricts into an automorphism of $\Ga_0$. 
This will imply the proposition.

For a group $G$ we say that $(A,B)$ is a \textit{decomposition pair} if $G=A\times B$.
We say that $(A,B)$ is a \textit{maximal decomposition pair} if given any decomposition pair $(A',B')$ we have that $A'$ or $B'$ is a subgroup of $A$. (Note that we do care about the order of $A$ and $B$ in this definition.)

{\bf Claim:} If $(A,B)$ is a maximal decomposition pair of $G$, then $A$ and $B$ are characteristic subgroups of $G$.

Indeed, consider a maximal decomposition pair $(A,B)$ and an automorphism $\theta\in\Aut(G).$
We have that $(\theta(A),\theta(B))$ is a decomposition pair of $G$ and likewise $(\theta^{-1}(A),\theta^{-1}(B))$ is a decomposition pair.
We then have four cases being all combinations of $\theta(Y),\theta^{-1}(Z)\subset A$ with $Y,Z\in\{A,B\}$.
Assume $Z=A$. We deduce that $\theta(Y)\subset A\subset \theta(A)$.
If $Y=A$, then $\theta(A)=A$ and we are done.
If $Y=B$, then $\theta(B)\subset\theta(A)$ implying that $B$ is trivial (since $A\cap B=\{e\}$) and thus $A=G$ and once again $A,B$ are characteristic subgroups (for trivial reasons).
The remaining cases with $Z=B$ can be treated similarly.

Set $G=\Ga\wr_{\Q_2}F$ and consider the pair $(\Ga_0,\Ga\wr_{U}F)$.
It is now sufficient to show that $(\Ga_0,\Ga\wr_{U}F)$ is a maximal decomposition pair of $G$.
We already observed that it is a decomposition pair so it remains to check it is maximal.
Consider another decomposition pair $(A,B)$ and the quotient map $q:\Ga\wr_{\Q_2}F\to F.$
We have that $q(A),q(B)$ are two subgroups of $F$ which generates $F$ and that mutually commute, i.e.~$q(a)q(b)=q(b)q(a)$ for all $a\in A, b\in B$.
Note that $q(A)\cap q(B)$ is a subgroup of $Z(F)$ the center of $F$. 
It is well-known that $Z(F)$ is trivial implying that $q(A)\cap q(B)$ is trivial.
We deduce that $(q(A),q(B))$ is a decomposition pair of $F$. 
However, it is easy to prove (and also well-known) that $F$ does not admit any non-trivial decomposition into a direct product of groups. 
We deduce that $q(A)$ or $q(B)$ is trivial. 
Up to swapping $A$ with $B$ we obtain that $q(A)$ is trivial and $q(B)=F$.

By definition of $q$ we deduce that $A\subset \oplus_{\Q_2}\Ga$. Moreover, note that for any $v\in F$ there exists $h_v\in \oplus_{\Q_2}\Ga$ satisfying $h_v v\in B$.
Consider $g\in A$ and its support $\supp(g)$ as a map from $\Q_2$ to $\Ga.$
Since $g$ commutes with all elements of $B$ we have that 
$$\supp(g) = \supp(\ad(h_v v)(g)):=\supp( h_v v g v^{-1} h_v^{-1}) = v\cdot \supp(g).$$
Therefore, $\supp(g)$ is stabilised by any element of $F$. 
Since $\supp(g)$ is necessarily finite and $\{0\}$ is the only finite $F$-orbit of $\Q_2$ we deduce that $\supp(g)\subset \{0\}$ for all $g\in A$.
This is equivalent to have $A\subset \Ga_0$. 
Therefore, $(\Ga_0,\Ga\wr_{U}F)$ is a maximal decomposition pair of $G$ and $\Ga_0$ is a characteristic subgroup of $G$.
Since $\Ga_0\simeq \Ga$ we easily deduce the statement of the proposition.

Consider now the group $X=T$ and the wreath product $G=K\rtimes T=\oplus_{\Q_2}\Ga\rtimes X$.
Let us prove that $K$ is a characteristic subgroup of $G$.
For any subgroup $H$ of $G$ we write $NC(H)$ for the \textit{normal closure} of $H$ inside $G$, i.e.~the smallest normal subgroup of $G$ containing $H$.
Say that $H\subset G$ has the decomposability property if there exists subgroups $A,B\subset G$ satisfying that 
$$H=A\times B \text{ and } NC(A)=NC(B)=H.$$
It is easy to see that $K\subset G$ has the decomposability property.
Let us show that any normal subgroup $H\lhd G$ with the decomposability property is contained in $K$. 
This will imply that $K$ is a characteristic subgroup (as being the largest normal subgroup of $G$ satisfying the decomposability property).
Assume $H=A\times B$, is a normal subgroup of $G$ and $NC(A)=NC(B)=H$.
Moreover, assume that $H\not\subset K$.
If $A\subset K$, then $NC(A)\subset K$ since $K$ is a normal subgroup implying $H\subset K$ and thus contradicting our assumption.
Therefore, $A,B$ are both not contained in $K$.
Consider the quotient map $q:G\to T$ and note that $q(H)\lhd T$ is a non-trivial normal subgroup since $H\not\subset K$ and $H\lhd G$ is normal.
Since $T$ is simple we deduce that $q(H)=T.$
Moreover, $q(A),q(B)$ are two subgroups of $T$ that are non-trivial, mutually commute, and generate $T$.
If $q(A)\cap q(B)$ is trivial, then $T=q(A)\times q(B)$ which is absurd since $T$ is simple.
Therefore, $q(A)\cap q(B)$ is non-trivial and moreover it is in the center of $T$ which again is absurd.
We have reached a contradiction and thus $H\subset K$ implying that $K$ is a characteristic subgroup.

Consider now an isomorphism $\theta:G\to \ti G$ where $G=\Ga\wr_{\Q_2} T, \ti G=\ti\Ga\wr_{\Q_2} T$ and $\Ga,\ti\Ga$ are (possibly twisted) wreath products.
From above, $\theta$ restricts into an isomorphism $\kappa:K\to\ti K$ where $K=\oplus_{\Q_2}\Ga, \ti K=\oplus_{\Q_2} \ti \Ga.$
We deduce that 
$$\theta(k v) = \kappa(k) \cdot c_v \cdot \phi_v \text{ for all } k\in K, v\in T$$
where $T\ni v\mapsto c_v\in K$ is a cocycle and $v\mapsto \phi_v$ an automorphism of $T$.
Using a theorem of Rubin, Brin showed that $\Aut(T)$ was isomorphic (in the obvious way) with the normaliser of $T$ inside the homeomorphism group of the unit circle $S_1$ \cite{Rubin89, Brin96}.
In fact, Brin proved the striking result that $\Aut(T)\simeq T\rtimes \Z/2\Z$ where $T$ is identified with its group of inner automorphisms and  where $\Z/2\Z$ is generated by an involution of the unit circle $S_1$ that necessarily normalises $T$.
In particular, there exists a homeomorphism $\varphi$ of $S_1$ that implements the automorphism $\phi$:
$$\phi_v=\varphi v\varphi^{-1} \text{ for all } v\in T.$$
From there it is not hard to conclude that necessarily $\varphi(\Q_2)=\Q_2$ and 
$$\supp(\theta(k))=\varphi(\supp(k)) \text{ for all } k\in K.$$
Therefore, by composing $\theta|_K$ with coordinate projections we obtain that for any $x\in\Q_2$ there exists $\kappa_x:\Ga\to \ti\Ga$ satisfying 
$$\kappa(k)(\varphi(x)) = \kappa_x(k(x)) \text{ for all } k\in K.$$
Since $\theta|_K$ is an isomorphism we necessarily have that $\kappa_x$ is an isomorphism for all $x\in \Q_2$ and in particular $\Ga$ and $\ti\Ga$ are isomorphic.
\end{proof}

From Theorems \ref{theo:FCWP} and  \ref{theo:classifWP} we deduce Corollary \ref{cor:B}. 
This permits to construct a large family of pairwise non-isomorphic FS groups using Ore HP monoids.

\begin{proof}[Proof of Corollary \ref{cor:B}]
Consider two FS groups $G,\ti G$ built from two Ore HP monoids $M,\ti M$, respectively.
If $G\simeq \ti G$, then $K_M\simeq K_{\ti M}$, by Theorems \ref{theo:FCWP} and \ref{theo:classifWP}.
The same holds for the $T$ and $V$-cases.

Assume now that we are in the untwisted case. Assume that $K_M\simeq K_{\ti M}$.
Then of course the {\it untwisted} wreath products $K_M\wr_{\Q_2} F$ and $K_{\ti M}\wr_{\Q_2}F$ are isomorphic.
Since these latter groups are isomorphic to $G$ and $\ti G$, respectively, we deduce that $G\simeq \ti G$.
A similar arguments provide the other cases.
\end{proof}

\subsection{Haagerup property}\label{sec:Haagerup}
A discrete group $G$ has the Haagerup property if there exists a net (a sequence when the group is countable) of positive definite maps $f_i:G\to\C$ that vanishes at infinity and converges pointwise to $1$. It is a weakening of amenabiity (take the same definition but replace ``vanishes at infinity'' by ``finitely supported''). 
It is equivalent to Gromov's a-T-menability: there exists a proper affine isometric action on a Hilbert space. 
Moreover, a group with this property satisfies the Baum-Connes conjecture (with coefficients) by a deep result of Higson and Kasparov \cite{Higson-Kasparov01}.
We refer the reader to the following book \cite{CCJJV01} and the recent survey of Valette for more details \cite{Valette18}.

Using the novel Jones' technology for constructing unitary representations the author proved that if $\Ga$ is any discrete group with the Haagerup property, then a twisted wreath product $\Ga\wr_{\Q_2}V$ has the Haagerup property (as a discrete group) for a large class of twists \cite{Brothier19WP}.
This generalises the previous result of Farley showing that $V$ has the Haagerup property \cite{Farley03} and extends a previous proof of Jones and the author \cite{Brothier-Jones19}. 
Moreover, this is the only analytic (not geometric) proof of such a statement for wreath products and the Haagerup property. 
A related statement using actions on walls is due to Cornulier, Stalder, and Valette where they obtain rather different examples \cite{Cornulier-Stalder-Valette12}.

\begin{corollary}
Let $M$ be a Ore HP monoid with associated FS category $\cF$ and FS groups $G^X$ with $X=F,T,V$.
We have that $G^X$ has the Haagerup property if and only if $\Frac(M)$ does.
\end{corollary}

\begin{proof}
Consider $M, G^V$ as above.
By Remark \ref{rem:twist} we have that $G^V\simeq H^V$ where $H^V$ is the group constructed from the pair $(K_M,\ad(a^{-1}))$ appearing in \cite{Brothier19WP}.
This latter group $H^V$ has been proven to have the Haagerup property if and only if $K_M$ has the Haagerup property in \cite{Brothier19WP}.
Observe that $\Frac(M)$ is an extension of $K_M$ by an abelian group. Hence, $\Frac(M)$ has the Haagerup property if and only if $K_M$ has this property.
We deduce that $G^V$ has the Haagerup property if and only if $K_M$ has this property.
The $F$ and $T$-case follow by noting that we have the tower of inclusions $$K_M\subset G\subset G^T\subset G^V$$ and by using that the Haagerup property is closed under taking subgroups.
\end{proof}

\subsection{Homological and topological finiteness properties}\label{sec:finiteness}

{\bf Definitions.}
Recall that a group $G$ is of type $FP_n$ for a given $n\geq 1$ if there exists a projective resolution of the $\Z[G]$-module $\Z$ (where $G$ acts trivially on it) such that the $n$ first modules are finitely generated. It is of type $FP_\infty$ if it is of type $FP_n$ for any $n\geq 1$.
This is a homological finiteness property.

We say that $G$ is of type $F_n$ if there exists a classifying space whose $n$-skeleton is finite. Similarly, $G$ is of type $F_\infty$ if it is of type $F_n$ for any $n\geq 1$.
This is a topological finiteness property.

By observing that the chain complex of a universal cover of a classifying space provides a free resolution of the $\Z[G]$-module $\Z$ we deduce that $F_n$ implies $FP_n$ for any $n\geq 1$.
Moreover, note that $F_1$ is equivalent to $FP_1$ which is itself equivalent to be finitely generated.
Condition $F_2$ is equivalent to be finitely presented. 
Now, if $G$ is finitely presented, then $FP_n$ is equivalent to $F_n$ for all $n$.
However, there exist groups of type $FP_2$ that are not finitely presented.
We refer the reader to the two excellent books of Brown and Geoghegan for details \cite{Brown82,Geoghegan-book}.

{\bf Spine and a previous theorem.}
It is a difficult task to decide if a group satisfies any of these properties. 
However, we proved that a large family of FS groups are of type $F_\infty$ \cite{Brothier22-FS}.
This is done using a theorem due to Thumann by considering the {\it spine} of the FS category \cite{Thumann17}.
If $\cF$ is a forest category with colour set $S$, then define $\Sp(\cF)_1$ to be the set of carets $\{Y_a:\ a\in S\}.$
Define $\Sp(\cF)_2$ to be the union $\cup_{a,b} \mcm(Y_a,Y_b)$ of the {\it minimal common right-multiples} of $Y_a$ and $Y_b$ so that $a\neq b$ are colours. 
Inductively construct $\Sp(\cF)_n$ by taking the union of the $\mcm(x,y)$ with distinct $x,y\in \Sp(\cF)_{n-1}.$

The {\it spine} of $\cF$ is
$$\Sp(\cF):=\bigcup_{n\geq 1} \Sp(\cF)_n.$$
If the spine of a Ore forest category is finite, then the associated FS groups $G^X$ with $X=F,T,V,BV$ are all of type $F_\infty$.
For instance, any FS presentation of the form $\la a,b| C_a(t)=C_b(s)\ra$ with $t,s$ monochromatic trees with same number of leaves provides a Ore FS category whose fraction group is of type $F_\infty$.
Although, in general we do not know which finiteness properties satisfy a generic FS group.

{\bf Specialisation of a previous result.}
We can define the spine for a HP monoid $M$ as well. Indeed, in the definition of above replace $\Sp(\cF)_1$ by the set of letters and then proceed as before for defining $\Sp(\cF)_n$.
In the specific situation of $\cF=\cF_M$ being a FS category constructed from $M$ we have that $\Sp(\cF)$ is in bijection with $\Sp(M)$. This permits to easily check if the spine of $\cF_M$ is finite or not.
Hence, we have the following corollary of our previous theorem on finiteness property and our Theorem \ref{theo:FCWP}.

\begin{corollary}
If $M$ is a Ore HP monoid with finite spine, then $G^X$ is of type $F_\infty$ for $X=F,T,V,BV$.
\end{corollary}

This corollary is obtained using a difficult topological proof based on the structure of a classifying space of the FS group $G^X$. 
It does have the advantage to work equally well in the twisted case (and thus proves that certain twisted wreath products are of type $F_\infty$ by Theorem \ref{theo:FCWP}).
However, it does not cover a number of cases of groups for which we believe are of type $F_\infty$. Additionally, it does not permit to consider other properties like $F_n$ or $FP_n$ with finite $n$.
We will see that other powerful techniques are available for the class of {\it wreath products}.

{\bf Result of Bartholdi, Cornulier, and Kochloukova.}
In our specific situation of FS groups obtained from HP monoids we can completely determine these properties in terms of the monoid.
Indeed, Bartholdi, Cornulier, and Kochloukova have (almost) completely settle the question for {\it untwisted} wreath products with the following result.
Say that a group action $W\act Q$ is $(P_m)$ if $Q$ is a nonempty set and for all $1\leq j\leq m$ we have that the diagonal action $W\act Q^j$ has finitely many orbits and if the stabiliser of a point is a subgroup of type $FP_{m-j}$.
Say that it is $(P_\infty)$ if it is $(P_m)$ for every $m\geq 1$.

\begin{theorem}\label{theo:BCK}\cite{Bartholdi-Cornulier-Kochloukova15}
Consider a group $K$, a group action $W\act Q$, and a natural number $m\geq 1$.
If $K$ is of type $FP_m$, $W$ of type $FP_m$, and $W\act Q$ is $(P_m)$, then the \emph{untwisted} wreath product $K\wr_Q W$ is of type $FP_m$.

Moreover, when $K$ has infinite abelianisation, then the converse is also true.
\end{theorem}

They deduce the topological analogues: take the statement of above and replace $FP_m$ by $F_m$ except for the condition on the stabilisers.
It is remarkable to have an ``if and only if'' statement (under the mild condition of having infinite abelianisation of $K$). Unsurprisingly, the most difficult part of their result resides in proving the converse.

Now, if $M$ is a Ore HP monoid, then it produces four FS groups $G^X$ with $X=F,T,V,BV$ using the FS category $\cF_M$ associated to $M$.
We have proven that $G^X\simeq K_M\wr_{\Q_2} X$ (where the wreath product may be twisted) where $X\act \Q_2$ is the usual action on the dyadic rationals and where $K_M$ is the kernel of the length map associated to the fixed presentation of $M$.
Recall that we have a (split) short exact sequence 
$$1\to K_M\to \Frac(M)\to \Z\to 1.$$
Since $\Z$ is a group of type $F_\infty$ (it even has a finite classifying space with two cells) we have that $K_M$ and $\Frac(M)$ have same finiteness properties.
In order to apply Theorem \ref{theo:BCK} it is then sufficient to prove that $X\act \Q_2$ is $(P_\infty)$ and to assume that $G^X$ is untwisted.

Bartholdi, Cornulier, and Kochloukova considered the exact same action $F\act \Q_2$ and observed it was $(P_\infty)$, see \cite[Section 6]{Bartholdi-Cornulier-Kochloukova15}.
Indeed, for any $m\geq 1$ the stabiliser of a point for $F\act\Q_2^m$ is isomorphic to a direct sum of $F$ and has finitely many orbits. 
Since $F$ is of type $F_\infty$ so does a finite direct sum of it.
If we replace $F$ by $T$, then again the stabiliser of a point for $T\act \Q_2^m$ is a finite direct sum of $F$.

However, for $V$ and $BV$ it is not that immediate to conclude.
Indeed, we have that $\Q_2^m$ has certainly finitely many orbits for the action of $V$ or $BV$ but we do not have a quick argument showing that the point stabilisers are of type $F_\infty$.
One can follow the topological proof showing that $V$ and $BV$ are of type $F_\infty$ and adjust it to deduce that the point stabilisers are of type $F_\infty$ \cite{Brown87,Bux-Fluch-Marschler-Witzel-Zaremsky16}. It is not hard but rather tedious. Especially for $BV$ where the proof is quite long and technical.
A cumbersome but quick way to check the veracity of these statement (i.e~$V\act \Q_2$ and $BV\act \Q_2$ are ($P_\infty$)) is as follows.
Take 
$$M=\Mon\la a,b| ab=ba\ra$$ 
which is a Ore monoid. We have that $\Frac(M)\simeq \Z\oplus \Z$ and $K_M\simeq \Z.$ In particular, $K_M$ has infinite abelianisation. 
Moreover, since $M$ is abelian, we have that $G^X$ is untwisted: it is isomorphic to the untwisted wreath product $\Z\wr_{\Q_2} X$.
We are then in the assumption of the theorem of Bartholdi, Cornulier, and Kochloukova.
Therefore, if $G^V:=\Z\wr_{\Q_2} V$ and $G^{BV}:=\Z\wr_{\Q_2} BV$ were of type $F_\infty$, then we would be able to conclude that the actions $V\act \Q_2$ and $BV\act\Q_2$ must necessarily be ($P_\infty$).
Now, $G^X$ for $X=V,BV$ are FS groups obtained from the FS category
$$\cF:=\FC\la a,b| a_{1,1}b_{1,2}=b_{1,1}a_{1,2}\ra.$$
The FS presentation has two colours and a single skein relation. Hence, its spine has only three elements (being $a_{1,1},b_{1,1}$ and $a_{1,1}b_{1,2}$). 
We can then conclude using our theorem (which uses Thumann's theorem on operad groups) that $G^X$ is of type $F_\infty$ for $X=V$ and $BV$. We then deduce that $V\act \Q_2$ and $BV\act \Q_2$ are $(P_\infty)$.
We obtain the following.

\begin{theorem}
Consider a Ore HP monoid $M$, its associated FS category $\cF$, and the FS groups $G^X$ with $X=F,T,V,BV$.
Assume that $G^X$ is untwisted.
Fix $m$ which is either a natural number or $\infty$. 
If $\Frac(M)$ is of type $FP_m$ or $F_m$, then so is $G^X$. 
Moreover, we have the converse when $K_M$ has infinite abelianisation.
\end{theorem}

{\bf Comparison of techniques.}
As mentioned above we previously proved that if $\cF$ has a finite spine, then the associated fraction groups $G^X$ with $X=F,T,V,BV$ are all of type $F_\infty$.
This theorem permits to prove that many FS groups are of type $F_\infty$ regardless of the existence or not of a decomposition in wreath products. 
The proof is entirely topological and relies on the study of the classifying space of $G$ deduced from the nerve of the category $\cF$.
In an article in preparation, Ryan Seelig and the author will demonstrate that this theorem using the spine covers a large family of virtually simple groups that are far from being wreath products \cite{Brothier-Seelig23}.

The result on wreath product has a homological proof. The one direction that interests us (if $M$ has type $FP_m$, then so does $G$) is rather short when compared with the aforementioned topological proofs. 
Moreover, it works for any sufficiently nice action $W\act X$ while here we only consider the very specific action of the Thompson groups on the dyadic rationals. 
However, this approach seems difficult to adapt to other classes of FS groups. 
It is not even clear to us how to adapt the homological proof to {\it twisted} wreath products such as the one appearing in this present article.
Although, we suggest new groups when the homological proof may be adapted at the end of this section.

{\bf A brief sketch of a topological proof.}
It is possible to prove topologically the following statement: if $\Gamma$ is a group of type $F_m$, then so is the possibly twisted wreath product $\Gamma\wr_{\Q_2}X$ for $X=F,T,V$.
Using our structural theorem this implies that if a Ore HP monoid $M$ is of type $F_m$, then the associated FS group 
$G$ (resp.~its $T,V$-versions) is of type $F_m$.
We briefly explain the idea of the proof. It adapts an unpublished strategy due to Tanushevski for similar (but non-isomorphic) groups that has been outlined by Witzel and Zaremsky in \cite[Section 6]{Witzel-Zaremsky18}.
Note that this method works equally in the twisted case.

{\it Another category.}
Fix a pair $(\Gamma,\alpha)$ where $\Gamma$ is a group and $\alpha$ an automorphism of $\Gamma$.
The pair plays the role of $(K_M,\ad(a^{-1}))$ where $M$ is a Ore HP monoid and $a$ a generator of $M$. 
We introduce the category $\cG$ of monochromatic forests but with leaves decorated by elements of $\Gamma$. 
Elements of $\cG$ are classes $[f,g]$ where $f$ is a forest and $g=(g_\ell)_\ell$ a map from the leaves $\ell$ of $f$ to $\Gamma$.
The equivalence relation on pairs $(f,g)$ is defined as follows.
Consider $\ti f$ which is the forest $f$ to which we add a caret say at the $k$-th leaf.
Now, decorate the $k$-th leaf of $\ti f$ by $\alpha(g_k)$ and the $(k+1)$-th by $e$ (the neutral element). Keep the other decorations untouched giving a new tuple $\ti g$.
We declare that $(f,g)$ and $(\ti f,\ti g)$ are in the same class and consider the equivalence relation generated by it. 
We compose as usual unlabelled forests by vertical stacking and for group elements of $\Gamma$ we consider the pointwise multiplication so that $[f,g] \circ [f, h]:=[f,gh]$.
This forms a Ore category $\cG$ with fraction groupoid $\Frac(\cG)$ equal to all $f\circ g\circ q^{-1}$ with $f,q$ forests with same number of leaves, say $n$, and $g\in \Gamma^n$ a decoration of the leaves.
The key observation is that the fraction group $\Frac(\cG,1)$ is isomorphic to the FS group $\Frac(\cF_M,1)$ when $(\Gamma,\alpha)=(K_M,\ad(a^{-1}))$.

{\it A classifying space.}
Let $E\subset \Frac(\cG)$ be the subset of the fraction groupoid of all $t\circ g\circ q^{-1}$ where $t$ is a {\it tree}, $g$ a tuple of group elements of $\Gamma$ and $q$ is a forest.
Consider now three actions on $E$:
$$\Frac(\cG,1)\act E \curvearrowleft \cUF$$
and $$E\curvearrowleft \cH.$$
All three actions are the restriction of the composition of elements of the groupoid $\Frac(\cG)$ to certain subsets.
The fraction group $\Frac(\cG,1)$ corresponds to all $t\circ g\circ s^{-1}$ where both $t$ and $s$ are trees.
The symbol $\cUF$ stands for forests with trivial decoration of the leaves and $\cH$ for the exact opposite: elements $g=I^{\ot n}\circ g\circ I^{\ot n}$ where $g\in \Gamma^n$ and where $I^{\ot n}$ is the trivial forest with $n$ roots and $n$ leaves.
Hence, $\cUF$ corresponds to the category of monochromatic trees and $\cH$ to the groupoid with object the natural number and isotropy groups being direct products of $\Gamma$.
The right action of $\cUF$ provides a partial order on $E$: $x\leq x\circ f$ for $x\in E$ and $f\in\cUF$.
This partial order is invariant under the action of $\Frac(\cG,1)$.
Finally, the right-action of $\cH$ obviously commutes with the left-action of $\Frac(\cG,1)$ and is compatible in a certain sense with the right-action of $\cUF$.
This permits to push the partial order of $E$ on the quotient $E/\cH$.
By taking the geometric realisation of the simplicial complex associated to the partially ordered set $E/\cH$ we obtain a $\Frac(\cG,1)$-space whose orbit space is a classifying space.

{\it Brown's criterion.}
By studying this specific classifying space and following a classical scheme based on Brown's criterion and Bestvina-Brady discrete Morse theory it is possible to conclude \cite{Brown87,Bestvina-Brady97}. We refer the reader to the expository article of Zaremsky where he explains very well this last step when applied to the simpler case of Thompson's group $F$ \cite{Zaremsky21}.

We deduce the following result.

\begin{theorem}\label{theo:F-twist}
Let $M$ be a Ore HP monoid with fraction group $\Frac(M)$ and associated FS groups $G^X$ with $X=F,T,V$.
Let $m$ be a natural number or $\infty$.
If $\Frac(M)$ is of type $F_m$, then so is $G^X$.
\end{theorem}

As far as we know, the unpublished theorem of Tanushevski proves the following.
The group $C(\cC,\Ga)\rtimes V$ is of type $F_m$ when $\Ga$ is (where $C(\cC,\Ga)$ stands for the group of continuous function from the Cantor space $\cC$ to $\Ga$). The original statement is quite different but we previously observed that his groups decompose in this specific form, see \cite{Brothier21}.
Above, we have briefly explained the argument not for $C(\cC,\Ga)\rtimes V$ but for the twisted wreath product $\Ga\wr_{\Q_2} F$.
It should be noted that these two classes of groups do not intersect as proved in \cite[Section 2.5]{Brothier21}.
It is not hard to go from the $F$-case to the $T$ and $V$-cases. 
It sounds reasonable to expect that the topological proof of Tanushevski can be adapted in the $BV$-case. Although, this would be highly technical and would demand its own detailed proof.
For this reason we left this case out of the statement of the theorem.

At the moment we don't know how to adapt the powerful homological techniques of Bartholdi, Cornulier, and Kochloukova to other FS groups.
Although, a first place to start would be to consider groups of the the form $C(Y,\Ga)\rtimes W$ constructed from a group $\Ga$ and an action by homeomorphisms $W\act Y$ on a topological space $Y$.

\subsection{Orderability}\label{sec:orderability}
In this short section we observe that many FS groups are orderable and bi-orderable. 
Thompson group $F$ is defined as a piecewise affine group of homeomorphism of the unit interval. In particular, it acts in an order-preserving way on the real line (which is equivalent to be orderable for countable groups). 
Most Thompson-like groups are built in a dynamical way by acting nicely on the unit interval or a similar structure. 
Our groups are rather different and are built {\it diagrammatically} rather than {\it dynamically}. We note here that not all of them will be orderable and thus will be very far from acting nicely on the real line.

{\bf Orderability.}
A left-ordered group is a pair $(\Ga,\leq)$ where $\Ga$ is a group and $\leq$ is a total order on $\Ga$ that is left-invariant: $h\leq k$ implies $gh\leq gk$ for all $h,k,g\in \Ga$. We say that $\leq$ is a left-order and $\Ga$ is left-orderable if it admits a left-order.
Similarly we define right-orders and say that $\leq$ is a bi-order if it left and right invariant.
A positive cone of $\Ga$ is a submonoid $P\subset \Ga$ satisfying that $P\cup P^{-1}=\Ga, P\cap P^{-1}=\{e\}$.
A left-order $\leq$ defines a positive cone $\{g\in\Ga:\ e\leq g\}$ and conversely a positive cone $P$ defines a left-order: $g\leq_P h$ if and only if $g^{-1}h\in P$.
For bi-orderability we further assume that $P$ is closed under conjugation.

Observe that bi-orderable implies left-orderable which itself implies torsion-free.

{\bf Permanence properties.}
By restriction of the order we have that a subgroup of a left-orderable group is left-orderable.
If $$1\to \Xi\to \Ga\to \La\to 1$$ is a split exact sequence, then $\Ga$ is left (resp.~bi-orderable) if and only if $\Xi$ and $\La$ are.
Indeed, one implication is obtained by taking the union of the positive cones of $\Xi$ and $\La$ giving a positive cone of $\Ga$. For the reverse implication observe that both $\Xi$ and $\La$ are subgroups of $\Ga$ since the exact sequence splits.

{\bf Orderable FS groups.}
Fix a Ore HP monoid $M$ and write $G^X$ for the associated $X$-FS group with $X=F,T,V,BV$.
We investigate when $G^X$ is left or bi-orderable.
Note that $T,V$ are not torsion-free and thus not left-orderable.
Moreover, since $X$ embeds in any $X$-forest group we know that $G^T$ and $G^V$ are never left-orderable. 
Hence, it only makes sense to consider $G$ and $G^{BV}.$

Thompson group $F$ is bi-orderable: a positive cone is given by all $g\in F$ having first nontrivial slope larger than $1$. (This order is obtained by pulling back the order of $[0,1]$ on which $F$ acts faithfully and in an order-preserving way.)
Moreover, the braided Thompson group $BV$ is left-orderable but not bi-orderable as proved first by Burillo-Gonzales-Meneses and later by Ishida using Witzel-Zaremsky's cloning systems \cite{Burillo-Gonzalez-Meneses08, Ishida18}.

Consider now our monoid $M$.
We have a split exact sequence
$$1\to K_M\to \Frac(M)\to \Z\to 1.$$
Since $\Z$ is clearly bi-orderable we have that $\Frac(M)$ is left or bi-orderable if and only if $K_M$ is.
Similarly, if $\cF$ is the FS category associated to $M$ and $G^X$ the associated FS group with $X=F,T,V,BV$, then we have a split exact sequence:
$$1\to \oplus_{\Q_2} K_M\to G^X\to X\to 1.$$
We deduce the following result.

\begin{theorem}
Let $M$ be a Ore HP monoid with fraction group $\Frac(M)$ and write $G, G^{BV}$ for the associated forest groups.
We have that:
\begin{enumerate}
\item $G$ is left or bi-orderable if and only if $\Frac(M)$ is;
\item $G^{BV}$ is left-orderable if and only if $\Frac(M)$ is.
\end{enumerate}
\end{theorem}

\section{The canonical action of a forest-skein group}\label{sec:canonical}

\subsection{Definition of the space and the action}\label{sec:def-Q}

We recall a construction appearing in \cite[Section 6]{Brothier22-FS}.
Given a FS category $\cF$ we consider the set of all pairs $(t,\ell)$ where $t$ is a tree of $\cF$ and $\ell$ a distinguished leaf of $t$.
We mod out this space by the equivalence relation $\sim$ generated by
$$(t,\ell)\sim (t\circ f, \ell^f)$$
where $f$ is a forest of $\cF$ (composable with $t$) and $\ell^f$ is the first leaf of the tree of $f$ that is attached to $\ell$.
Hence, if $f$ has a single caret at the $\ell$-th root, then $\ell^f$ is the left-leaf of this caret.
We obtain a space $Q=Q_\cF$ with classes denoted $[t,f]$.

Assume now that $\cF$ is a Ore category and denote by $G^F,G^T,G^V$ the associated FS groups.
There is a unique total order $\leq$ on $Q$ satisfying that $[t,k]\leq [t,\ell]$ if $k\leq \ell$. 
We have an action $\al:G^V\act Q$ satisfying that $G^F$ acts in an order-preserving way. 
Moreover, the restriction of $\al$ to $G^T$ is conjugated to the homogeneous action $G^T\act G^T/G^F$.
The action is described by the formula:
$$\al([s \circ \pi,t]) [t,\ell]:= [s, \pi(\ell)]$$
where $x,y,t$ are trees, $\ell$ a leaf of $t$ and $\pi$ a permutation of the leaves of $t$.

The $G^V$-space $Q$ is important for studying FS groups. It is a replacement of the usual action of the Thompson groups on the dyadic rationals. 
It was used to obtain a topological finiteness permanence property \cite{Brothier22-FS}.
Moreover, in a work in progress of Ryan Seelig and the author we use $Q$ for proving that certain FS groups are simple \cite{Brothier-Seelig23}.
This later uses in a crucial way that $G\act Q$ is faithful for certain $G$ (which implies that $G$ is left-orderable).
When first introduced, the author naively believed that the action $G\act Q$ was {\it always} faithful regardless of which FS group $G$ is chosen. 
We will see that this is far from being the case.

\subsection{A non-faithful action}

Consider a Ore forest-skein category $\cF$ with a shape-preserving skein relation and fix a colour $a$ of $\cF$.
Following Section \ref{sec:def-Q}, we obtain a group isomorphism
$$\theta:G^V\to K\rtimes V$$
such that $K$ is the subgroup of all $[t,t']\in G$ satisfying that $t$ and $t'$ have same shape and where the action $V\act K$ is given by
$$v\cdot k:= C_a(v) k C_a(v)^{-1}$$
where $C_a:V\to G^V$ is the map consisting in colouring all interior vertices by $a$.

We now describe $G^T\act Q$ by describing $Q$ via the homogeneous space $G^T/G^F$.
Observe that $\theta(G^T)=K\rtimes T$ and $\theta(G)=K\rtimes F$.
We deduce that the action $G^T\act G^T/G^F$ is conjugated to the action 
$$K\rtimes T\act T/F,\ kg \cdot xF:= gxF.$$ 
Hence, $T$ acts on $T/F$ in the obvious homogeneous way and $K$ acts trivially.

Here is another way to understand why $K$ acts trivially. 
Describe $Q$ as the quotient space of pairs $[t,\ell]$ as above.
Consider any tree $t$ of $\cF$ and a leaf $\ell$ of it.
Now, we consider $C_a(t)$ the tree with same shape than $t$ but with all vertex coloured by $a$.
Since $\cF$ satisfies Ore's property there exists some forests $f,h$ so that $t\circ f = C_a(t)\circ h$.
At the level of $Q$ we deduce that 
$$(t,\ell) \sim (t\circ f, \ell^f) = (C_a(t)\circ h, \ell^f).$$
Observe that $\ell^f = \ell^h$ since $t$ and $C_a(t)$ have same shape.
This implies that 
$$(C_a(t)\circ h, \ell^f) = (C_a(t)\circ h, \ell^h) \sim (C_a(t),\ell).$$
Therefore, $$[t,\ell]=[C_a(t),\ell].$$ 
This means that $[t,\ell]$ only remembers the shape of $t$ and not its colouring.
Hence, $Q$ is in bijection with the set of classes $[s,\ell]$ with $s$ a {\it monochromatic} tree and $\ell$ a leaf of $s$.
This later is in bijection with $T/F$ itself in bijection with the dyadic rationals $\Q_2:=\Z[1/2]\cap [0,1)$. 
We deduce a $G^V$-equivariant bijection 
$$\theta:Q\to \Q_2$$
for the obvious action $$G^V\onto V\to \Aut(\Q_2)$$ with second arrow being the classical action of $V\act \Q_2$.

In particular we obtain the following statement.

\begin{proposition}
Consider a Ore FS category $\cF$ with a shape-preserving skein presentation. 
Denote by $G^V$ its associated $V$-group and by $\alpha:G^V\act Q$ the action described in Section \ref{sec:def-Q}.
If $D:G^V\to V$ is the decolouring map, then $\ker(\al)=\ker(D).$
\end{proposition}

Note that if $M$ is a Ore HP monoid and $\cF$ the associated FS category, then $\cF$ is in particular shape-preserving. Hence, the associated group $G^V$ satisfies the assumption of above.
Moreover, we obtain that $G^V\simeq \ker(D)\rtimes V$ where $\ker(D)\simeq\oplus_{\Q_2}K_M$.
Therefore, the normal subgroup $\oplus_{\Q_2}K_M$ always acts trivially on the space $Q$ and is equal to the kernel of the action $\alpha$.

\section{Examples}\label{sec:examples}
In this section we construct various examples of FS groups using monoids. 
This produces a bank of examples and counterexamples for a number of properties. 
In all these section $M$ denotes a HP monoid, $\cF_M$ the associated FS category, and $G^X$ (when they exist) the associated  FS groups where $X=F,T,V,BV$. 
As usual we write $G$ for $G^F$ and $K_M$ for the kernel of the length map $\Frac(M)\to \Z$ so that $G^X\simeq K_M\wr_{\Q_2}X.$

\subsection{Properties shared by all forest groups}
We remind the reader that  any FS group $G^X$ contains a copy of $X$ for $X=F,T,V,BV$, see \cite[Section 3.3]{Brothier22-FS}.
This implies for instance that $G^X$ is not elementary amenable, has exponential growth, and infinite geometric dimension (every classifying space is infinite dimensional).
Moreover, any FS group has trivial first $\ell^2$-Betti number, see \cite[Section 5]{Brothier22-FS}.

\subsection{Properties shared by all shape-preserving forest groups}
As mentioned earlier if $G^X$ is the FS group of a shape-preserving FS category, then $X$ is a quotient of $G^X$ for $X=F,T,V,BV$.
Note that all three groups $F,T,V$ have the Haagerup property and thus don't have Kazhdan property (T).
Since this latter property is closed under taking quotient we deduce that $BV$ does not have Kazhdan property (T) and none of the $G^X$ with $X=F,T,V,BV$.

\subsection{Torsion and center}
Thompson group $F$ is torsion-free, is ICC (i.e~has infinite conjugacy classes), and has a trivial centre. 
Consider $$M=\Mon\la a,b| aa=bb, ab=ba\ra$$ which is a Ore HP monoid. 
We have that $\Frac(M)\simeq \Z\oplus\Z/2\Z$ and $K_M\simeq \Z/2\Z.$
Therefore, $G^X\simeq \Z/2\Z\wr_{\Q_2} X.$
In particular, $G$ has torsion and has a nontrivial center isomorphic to $\Z/2\Z$ (and in particular is not ICC).
Similar construction provides infinite center: take for instance $\ti M:=\Mon\la a,b| ab=ba\ra$ noting that $K_{\ti M}\simeq \Z$ and thus giving that $Z(G_{\ti M})\simeq \Z.$

\subsection{Containment of a free group and amenability}
A famous result of Brin and Squier proves that the free group $\mathbb F_2$ in two generators does not embed inside $F$ \cite{Brin-Squier85}.
Consider the braid monoid over $n\geq 2$ strands
$$B_n^+:=\Mon\langle a_1,\cdots,a_{n-1} | a_j a_{j+1} a_j = a_{j+1} a_j a_{j+1},\ a_k a_m = a_m a_k, 1\leq j\leq n-2, |k-m|\geq 2\rangle.$$
This is a Ore HP monoid with fraction group $B_n$: the Artin braid group. 
For $n\geq 3$ note that $B_n$ contains a copy of $\mathbb F_2$ and so does the kernel of the length map $K_n$.
This implies that the associated FS group $G_n$ (the $F$-version) contains a copy of $K_n$ and thus of $\mathbb F_2$.
In particular, there exist $F$-versions of FS groups that are not amenable.

\subsection{Generators}
We have proven that if $\cF$ is a Ore FS category with a finite FS presentation, then $G$ is a finitely presented group. We even gave a bound on the number of generators and relations of $G$ in terms of the number of colours and skein relations.
Here is an easy example showing that the converse does not hold.
Consider the Thompson monoid with its infinite presentation:
$$F^+:=\la x_1,x_2,\cdots| x_qx_j=x_j x_{q+1}, 1\leq j<q\ra.$$
It is of course a Ore HP monoid with fraction group the Thompson group $F$.
The associated forest category $\cF_{F^+}$ has infinitely many colours and skein relations.
However, the FS group $G$ is finitely presented and in fact of type $F_\infty$ since $F$ is (via Theorem \ref{lettercor} item 3).

\subsection{Spine}
Consider again $$M=\la a,b| aa=bb, ab=ba\ra$$ with associated $\cF_M$ and $G^X$.
Note that $M$ is abelian implying that $G^X$ is untwisted.
Using the result of Bartholdi, Cornulier, and Kochloukova we deduce that $G^X\simeq\Z/2\Z\wr_{\Q_2} X$ is of type $F_\infty$ since $\Z/2\Z$ is \cite{Bartholdi-Cornulier-Kochloukova15}.
However, note that the spine of $M$ is infinite and thus so is the spine of $\cF_M$ preventing us to use our topological result of \cite[Section 7]{Brothier22-FS}.
It is rather surprising that such an easy group to apprehend is resisting to our previous theorem on general FS groups.

The monoid $M$ is an example of a {\it thin} monoid in the sense of Dehornoy \cite{Dehornoy02}.
It is not a Garside monoid but still satisfies nice properties concerning existence of common multiples.
Dehornoy provided two additional examples of thin monoids (that were Ore and HP) that happened to have {\it infinite} spines.
They are:
$$\Mon\langle a,b,c | a^2=b^2=c^2, ab=bc=ca, ac=ba=cb\rangle$$
and
$$\Mon\langle a,b,c| ac=ca=b^2, ab=bc, cb=ba\rangle.$$
It would be interesting to determine which properties their fraction groups and associated FS groups satisfy.

\subsection{Haagerup property}
All Thompson's groups $F,T,V$ are known to satisfy the Haagerup property (i.e.~Gromov's a-T-menability) by Farley \cite{Farley03}.
Moreover, a number of Thompson-like groups share this property like Guba-Sapir's diagram groups over semigroup presentations (and their annular and $V$-versions called braided) \cite{Guba-Sapir97} as proved by Farley \cite{Farley05}.
The same is true for the class of finite similarity structure groups of Hughes \cite{Hughes09}.
Note that these two classes of groups have a large intersection, see \cite{Farley-Hughes17}.

In general, this kind of result is proven by constructing a rather natural CAT(0) cubical complex on which the group acts properly by isometries. 
The complex usually comes from the universal cover of a classifying space of the group.
We could not find any similar structure to the nerve of a FS category (which is a universal cover of a classifying space) or any complex derived from it.
We now seriously doubt that such a CAT(0) cubical complex exists for each FS group.
Indeed, this would imply that if $M$ is a Ore HP monoid, then its fraction group $\Frac(M)$ has the Haagerup property. 
In particular, this would imply that all braid groups have the Haagerup property which is still unknown today.

Although, we still don't know any example of Ore HP monoid $M$ for which $\Frac(M)$ has been proven to not have the Haagerup property nor any FS groups.

\subsection{Finiteness properties}
The three Thompson groups $F,T,V$ and the braided version $BV$ are of type $F_\infty$.
Many generalisations of these groups share this property.
We have shown that in many cases $\Frac(M)$ and $G^X$ share the same finiteness properties.
However, we still don't know if all finitely presented FS groups are automatically of type $F_\infty$.
In particular, we don't have any example of a Ore HP monoid $M$ with $\Frac(M)$ finitely presented not of type $F_\infty$.

\subsection{Orderability}
As previously observed all $G^T$ and $G^V$ are not left-orderable since they have some torsion.
Now, we can observe that there exist some FS groups that are bi-orderable (take $G=F$ or $G=\Frac(\cF_M,1)$ with $M=\N\times\N$ and usual presentation), left-orderable but not bi-orderable (take $M$ equal to a braid group over $n\geq 3$ strands with its usual presentation), or not left-orderable by taking $M=\Mon\la a,b| ab=ba,aa=bb\ra$ since the associated FS group has torsion.
Taking the same monoids we obtain some $BV$-FS groups that are left-orderable and other that are not.

\subsection{Derived subgroup and simplicity}
The Thompson groups and related groups are famous because they produce simple groups.
Indeed, the derived subgroup $\Delta(F)$ of $F$ is simple and $T,V$ are both simple.
Ryan Seelig and the authors are studying a large class of FS groups whose derived groups are simple (and in particular the $T$ and $V$-versions are virtually simple) \cite{Brothier-Seelig23}.
A shape-preserving Ore FS category produces groups of the form $G^X\simeq \ker(D)\rtimes X$ where $D$ is the decolouring map.
Hence, it is far from being simple or having simple derived subgroups.
If $\cF$ is constructed from a monoid, then $G^X$ is a wreath product with $\ker(D)\simeq \oplus_{\Q_2}K_M$.
Using this wreath product decomposition it is now easy to construct FS groups for which each subgroup of the derived series is not simple.

\subsection{Iteration}
We have proven that if $M^{(1)}$ is a Ore HP monoid, then the group $G^{(2)}:=\Frac(\cF^{(2)},1)$ is isomorphic to the fraction group of the FS monoid $\Frac(\cF^{(2)}_\infty)$ where $\cF^{(2)}$ is the FS category associated to $M^{(1)}$.
We have observed that $\cF^{(2)}_\infty$ is itself a Ore HP monoid.
(Recall that if $\cF$ is an FS category, then $\cF_\infty$ is the associated FS {\it monoid} made of forests with infinitely many roots but having only finitely many nontrivial trees, see \cite[Section 1.3.2]{Brothier22-FS} for details.)
Hence, we can reiterate the process by considering $\cF^{(3)}$: the FS category associated to the Ore HP monoid $M^{(2)}:=\cF^{(2)}_\infty$.
We obtain a chain of groups $G^{(n)}$.
They can be expressed by iterated wreath products:
$$G^{(n+1)}\simeq K^{(n)}\wr_{\Q_2}F$$
where $K^{(n)}$ is the kernel of the length morphism $G^{(n)}\onto\Z.$
Although, these groups can also be described using the group presentation innerated from a skein presentation.
For instance, consider $M^{(1)}:=F^+$ the Thompson monoid with usual infinite presentation
$$\Mon\la x_n, n\geq 1| x_qx_j=x_jx_{q+1}, 1\leq j<q\ra.$$
We obtain that $G^{(2)}$ has a presentation with generators $x_{i,j}$ with two indices $i,j\geq 1$ satisfying the following Thompson-like relations:
$$x_{q,n} x_{j,n} = x_{j,n} x_{q+1,n} \text{ for all } n\geq 1 \text{ and all } 1\leq j<q$$
and 
$$x_{k,q} x_{r,j} = x_{r,j} x_{k,q+1} \text{ for all } k,r\geq 1 \text{ and all } 1\leq j<q.$$
Similarly, $G^{(n)}$ admits a generator set of $x_{i_1,\cdots,i_n}$ with $i_1,\cdots,i_n\geq 1$ satisfying similar Thompson-like relations.


\newcommand{\etalchar}[1]{$^{#1}$}

\end{document}